\documentclass[reqno,a4paper]{amsart}
\usepackage[latin1]{inputenc}
\usepackage[english]{babel}
\usepackage{eucal,amsfonts,amssymb,amsmath,amsthm,epsfig,mathrsfs}
\usepackage{dsfont,cancel,soul}
\usepackage{color}
\usepackage{amscd,amsxtra}
\usepackage{enumerate}
\usepackage{latexsym}
\usepackage{bm}

\allowdisplaybreaks

\newcounter{ipotesi}

\Alph{ipotesi}
 \makeatletter \@addtoreset{equation}{section}

\makeatother \makeatletter

\newtheorem{thm}{Theorem}[section]
\newtheorem{hyp}[thm]{Hypotheses}{\rm}
{\rm}
\newtheorem{lemm}[thm]{Lemma}
\newtheorem{coro}[thm]{Corollary}
\newtheorem{prop}[thm]{Proposition}

\newtheorem{rmk}[thm]{Remark}{\rm}
\newtheorem{example}[thm]{Example}

\newcounter{parentenv}

\newcommand{\R}{{\mathbb R}}

\newcommand{\N}{{\mathbb N}}

\newcommand{\Rd}{\mathbb R^d}
\newcommand{\Rn}{\mathbb R^n}
\newcommand{\Rm}{\mathbb R^m}

\newcommand{\T}{{\bm T}}

\newcommand{\e}{{\bm e}}
\newcommand{\g}{{\bm g}}

\newcommand{\f}{{\bm f}}
\newcommand{\bb}{{\bm b}}
\newcommand{\uu}{{\bm u}}
\newcommand{\A}{\bm{\mathcal A}}
\newcommand{\1}{\mathds 1}

\newcommand{\vv}{{\bm v}}

\newcommand{\zz}{{\bm z}}
\newcommand{\one}{\mbox{$1\!\!\!\;\mathrm{l}$}}

\newcommand{\ra}{\rightarrow}

\newcommand{\G}{{\bm G}}

\newcommand{\ol}[1]{\overline{#1}}

\renewcommand{\hat}[1]{\widehat{#1}}
\renewcommand{\tilde}[1]{\widetilde{#1}}

\newcommand{\set}[1]{{\left\{#1\right\}}}
\newcommand{\pa}[1]{{\left(#1\right)}}

\newcommand{\abs}[1]{{\left|#1\right|}}
\newcommand{\norm}[1]{{\left\|#1\right\|}}

\newcommand{\eqsys}[1]{{\left\{\begin{array}{ll}#1\end{array}\right.}}
\newcommand{\tc}{\, \middle |\,}

\begin{document}

\title[On coupled systems of PDEs with unbounded coefficients]{On coupled systems of PDEs\\ with unbounded coefficients}
\thanks{The authors are members of G.N.A.M.P.A. of the Italian Istituto Nazionale di Alta Matematica (INdAM).
Work partially supported by the INdAM-GNAMPA Project 2017 ``Equazioni e sistemi di
equazioni di Kolmogorov in dimensione finita e non''.}
\author[L. Angiuli and L. Lorenzi]{Luciana Angiuli, Luca Lorenzi}
\address{L.A.:  Dipartimento di Matematica e Fisica ``Ennio De Giorgi'', Universit\`a del Salento, Via per Arnesano, I-73100 LECCE (Italy)}
\address{L.L.: Dipartimento di Matematica e Informatica, Universit\`a degli Studi di Parma, Parco Area delle Scienze 53/A, I-43124 PARMA (Italy)}
\email{luciana.angiuli@unisalento.it}
\email{luca.lorenzi@unipr.it}

\date{}

\keywords{Nonautonomous parabolic systems, unbounded coefficients,
evolution operators, compactness, invariant subspaces, evolution systems of invariant measures}
\subjclass[2000]{35K40, 35K45, 37L40, 46B50, 47A15}

\begin{abstract}
We study the Cauchy problem associated to parabolic systems of the form $D_t\uu=\A(t)\uu$ in $C_b(\Rd;\Rm)$, the space of continuous and bounded functions $\f:\Rd\to\Rm$. Here $\A(t)$ is a weakly coupled elliptic operator acting on vector-valued functions, having diffusion and drift coefficients which change from equation to equation. We prove existence and uniqueness of the evolution operator $\G(t,s)$ which governs the problem in $C_b(\Rd;\Rm)$ proving its positivity. The compactness of $\G(t,s)$ in $C_b(\Rd;\Rm)$ and some of its consequences are also studied. Finally, we extend the evolution operator $\G(t,s)$ to the $L^p$- spaces related to the so called "evolution system of measures" and we provide conditions for the compactness of $\G(t,s)$ in this setting.
\end{abstract}

\maketitle

\section{Introduction}
In the study of the diffusion processes, second-order elliptic operators with unbounded coefficients appear naturally and the associated parabolic equation represents the Kolmogorov equation of the process. The theory of such equations is now well developed in the scalar case as the systematic treatise of \cite{newbook} and the reference therein show. On the contrary, the literature on systems of parabolic equations with unbounded coefficients is at a first stage and only some partial results are available.
The interest in the study of systems is on one hand motivated by the natural sake of extending the known results of the scalar case. On the other hand, systems of parabolic equations with unbounded coefficients arise in many applications. Among them we quote the study of backward-forward stochastic differential systems, the study of Nash equilibria to stochastic differential games, the analysis of the weighted $\overline{\partial}$- problem in $\mathbb{C}^d$, in the time-dependent Born-Openheimer theory and also in the study of Navier-Stokes equations. We refer the reader to \cite[Section 6]{AALT} and \cite{BeGoTe, dallara,HanRha,haslinger,HRS, HRS1}.

One of the first papers concerning parabolic systems with unbounded coefficients is \cite{HLPRS} where the authors prove that the realization $\A_p$ of the weakly coupled elliptic operator $\mathcal{\A}\uu={\rm div}(Q\nabla \uu)+F\cdot \nabla \uu+C\uu$ in $L^p(\Rm;\Rm)$ generates a strongly continuous semigroup and they characterize its domain under suitable assumptions on its coefficients. More precisely, they assume that the diffusion coefficients $Q=(q_{ij})$ are uniformly elliptic and bounded together with their first-order derivatives, the drift coefficient $F$ and the potential $V$ are sufficiently smooth and allow to grow as $|x|\log|x|$ and $\log|x|$, respectively, as $|x|\to +\infty$.

Next, first in \cite{DelLor11OnA} (in the weakly coupled case) and then in \cite{AALT} (also in the nonautonomous case), systems of parabolic equations with unbounded coefficients coupled up to the first order have been studied in the space of bounded and continuous functions over $\Rd$, and existence and uniqueness results for a classical solution to the associated Cauchy problem are established. This allows to introduce a vector-valued semigroup $\T(t)$ (an evolution operator $\G(t,s)$ in the nonautonomous case) in $\mathcal{L}(C_b(\Rd;\Rm))$ with the operator $\A(t)$.

Taking advantage of the results in \cite{AALT}, the authors of \cite{AngLorPal} provide sufficient conditions for the semigroup $\T(t)$ to admit a bounded extension to $L^p(\Rd;\Rm)$. Also some summability improving properties of the semigroup are studied. More precisely, hypercontractivity estimates of the form $\|\T(t)\|_{\mathcal{L}(L^p(\Rd;\Rm), L^q(\Rd;\Rm))}\le c_{p,q}(t)$ for any $1\le p\le+\infty$ and some positive function $c_{p,q}:(0,+\infty)\to (0,+\infty)$ are established. We stress that also the nonautonomous case is considered in \cite{AngLorPal}.

All the above papers have a common feature: the elliptic operators therein considered have all the diffusion coefficients that do not change from equation to equation, i.e.,
\begin{eqnarray*}
(\A_0\uu)_k= {\rm Tr}(QD^2u_k)+\sum_{i=1}^d (B_i D_i \uu)_k+ (C\uu)_k, \qquad\;\,k=1, \ldots, m.
\end{eqnarray*}
This form of the equations allows to extend easily the classical maximum principle for systems with bounded coefficients, which in turn allows to prove the  uniqueness of the classical solution of the Cauchy problem associated with the operator $\A_0$ and provides a comparison between the vector-valued semigroup $\T(t)$ associated to $\A_0$ and the scalar semigroup $T(t)$ associated to the operator $\mathcal{A}={\rm Tr}(QD^2)+\langle b,\nabla\rangle$
for a suitable drift term $b$, i.e., it can be shown that there exists $K\in \R$ such that
\begin{eqnarray*}
|\T(t)\f|^2\le e^{K(t-s)}T(t)|\f|^2,\qquad\;\, \f\in C_b(\Rd;\Rm),\qquad\; \,t>0.
\end{eqnarray*}
This is also the case considered in \cite{AAL_Inv1} where the matrices $B_i$ split in two terms: the leading one which is of diagonal type (like in the weakly coupled case) and the other one whose growth at infinity is controlled by a power of the minimum eigenvalue of the diffusion matrix.

In this paper, differently from the cases so far considered, we deal with nonautonomous weakly coupled operators with diffusion and drift coefficients which may vary from equation to equation, acting on a smooth function $\bm\psi$ as follows
\begin{eqnarray*}
(\A(t)\boldsymbol{\psi})_k(t,x)= {\rm Tr}(Q^k(t,x)D^2\psi_k(x))+\langle {\bf b}^k(t,x),\nabla \psi_k(x)\rangle+(C(t,x)\psi(x))_k,
\end{eqnarray*}
for any $(t,x)\in I\times \R$ and $k=1,\ldots, m$, $I$ being a right halfline (possibly $I=\R$).
The form of the operator $\A(t)$ makes the associated Cauchy problem
\begin{equation}\label{pb}
\left\{
\begin{array}{ll}
D_t\uu=\A(t)\uu, &{\rm in}~(s,+\infty)\times\R^d,\\[1mm]
\uu(s,\cdot)=\f\in C_b(\Rd;\Rm), &{\rm in}~\Rd,
\end{array}
\right.
\end{equation}
quite involved. In particular, in this case we are not able to control the solution of problem \eqref{pb} in terms of a scalar semigroup. To overcome this difficulty we extend to our situation a maximum principle for systems having bounded coefficients to the case of unbounded coefficients assuming that the off-diagonal entries of the matrix $C$ are bounded from below and the sum of each row of the matrix $C$ is bounded from above.
This yields the uniqueness of the classical solution to problem \eqref{pb}.

Once uniqueness is guaranteed, the existence of a classical solution of the problem (1.2) is then proved by some
compactness and localization argument based on interior Schauder estimates recalled in the Appendix. As a byproduct, we can associate an evolution operator $\G(t,s)$ to $\A(t)$ in $C_b(\Rd;\Rm)$, in the natural way.

The evolution operator $\G(t,s)$ is positive if the off-diagonal entries of $C$ are nonnegative and the system does not contain any subsystem which decouple, then each component of $\G(\cdot,s)\f$ is strictly positive in $(s,+\infty)\times\Rd$ whenever $\f$ is a nonnegative function which has at least a component that does not identically vanish in $\Rd$.

In \cite{AALT} the authors study the compactness of the evolution operator $\G_0(t,s)$ ($t>s\in I$) in $\mathcal{L}(C_b(\Rd;\Rm))$ showing that it is equivalent to the tightness of the measures $\{|p_{ij}(t,s,x,\cdot)|:\, x\in \Rd\}$ for any $i,j=1,\ldots,m$, where $p_{ij}(t,s,x,\cdot)$ are the transition kernels associated to the problem, i.e., for any $\f\in C_b(\Rd;\Rm)$, $s\in I$ and $k=1,\ldots, m$
\begin{eqnarray*}
(\G(t,s)\f)_k(x)=\sum_{i=1}^m \int_{\Rd}f_i(y)p_{ki}(t,s,x,dy),\qquad\;\,(t,x)\in (s,+\infty)\times \Rd.
\end{eqnarray*}
This fact together with the pointwise estimate of $|\G(t,s)\f|^2$ in terms of the scalar evolution operator associated to the operator $\A$, guarantees that the compactness of the scalar evolution operator is a sufficient condition for the compactness of $\G(t,s)$, hence the problem reduces to find conditions that ensure compactness in the scalar case.
We prove that, also in our case, the compactness of $\G(t,s)$ is equivalent to the tightness of the transition kernels associated to the problem (which are nonnegative measures if the off-diagonal entries of $C$ are nonnegative). On the other hand, the lack of a scalar evolution operator which ``dominates'' $\G(t,s)$ prevents us from applying the results of the scalar case.
However, it is possible to provide sufficient conditions for the compactness of $\G(t,s)$ in $C_b(\Rd;\Rm)$ in terms of the existence of some Lyapunov functions, see Theorem \ref{thm_comp_2}. In this case $\G(t,s)$ preserves neither $C_0(\Rd;\Rm)$ nor $L^p(\Rd;\Rm)$ for $p\in [1,+\infty)$.
Further, assumptions on the coefficients of $\A$ are provided which guarantee that these spaces together with the space $C^1_b(\Rd;\Rm)$ are preserved by the action of $\G(t,s)$.

Finally, we prove the existence of an evolution system of measures  associated with the evolution operator $\G(t,s)$ consisting of positive measures (which are equivalent to the Lebesgue one), where, according to the definition introduced in \cite{AAL_Inv,AAL_Inv1}, a family $\{\mu_{i,t}:\,t\in I,\,i=1,\ldots,m\}$ is an evolution system of measures if
\begin{eqnarray*}
\sum_{j=1}^m\int_{\Rd}(\G(t,s)\f)_jd\mu_{j,t}=\sum_{j=1}^m\int_{\Rd}f_jd\mu_{i,s},\qquad\;\,I\ni s<t,
\end{eqnarray*}
for any $\f=(f_1,\ldots,f_m)\in C_b(\Rd;\Rm)$,
where $(\G(t,s)\f)_j$ denotes the $j$-th component of the vector-valued function $\G(t,s)$.
We prove that the evolution operator $\G(t,s)$ can be extended with a bounded operator mapping $L^p_{{\bm\mu}_s}(\Rd;\Rm)$ into $L^p_{{\bm\mu}_t}(\Rd;\Rm)$ for any $p \in [1, +\infty)$ and provide sufficient conditions to be compact from $L^p_{{\bm\mu}_s}(\Rd;\Rm)$ into $L^p_{{\bm\mu}_t}(\Rd;\Rm)$ for any $p \in (1, +\infty)$.

\subsection*{Notation.}
Vector-valued functions are displayed in bold style. Given a function $\f$ (resp. a sequence
$(\f_n)$) as above, we denote by $f_i$ (resp. $f_{n,i}$) its $i$-th component (resp. the $i$-th component
of the function $\f_n$).
By $B_b(\Rd;\Rm)$ we denote the set of all the bounded Borel measurable functions $\f:\Rd\to\Rm$, where $\|\f\|_{\infty}^2=\sum_{k=1}^m\sup_{x\in\Rd}|f_k(x)|^2$.
For any $k\ge 0$, $C^k_b(\R^d;\R^m)$ is the space of all $\f:\Rd\to\Rm$
whose components belong to $C^k_b(\R^d)$, where the notation $C^k(\R^d)$ ($k\ge 0$) is standard and we use the subscripts ``$c$'', ``$0$'' and ``$b$'', respectively, for spaces of functions with compact support, vanishing at infinity and bounded.
Similarly, when $k\in (0,1)$, we use the subscript ``loc'' to denote the space of all $f\in C(\Rd)$
which are H\"older continuous in any compact set of $\Rd$.
We assume that the reader is familiar also with the parabolic spaces $C^{\alpha/2,\alpha}(I\times \Rd)$
($\alpha\in (0,1)$) and $C^{1,2}(I\times \Rd)$, and we use the subscript ``loc'' with the same meaning as above.

The symbols $D_tf$, $D_i f$ and $D_{ij}f$, respectively, denote the time derivative, the first-order spatial derivative with respect to the $i$-th variable and
 the second-order spatial derivative with respect to the $i$-th and $j$-th variables.
We write $J_x\uu$ for the Jacobian matrix of $\uu$ with respect to the spatial variables, omitting the subscript $x$ when no confusion may arise.
By $\boldsymbol{e}_j$ we denote the $j$-th vector of the Euclidean basis of $\R^m$. $\bm{\one}$ (resp. $\bm{0}$) denotes the $m$-valued function with entries all equal to $\one$ (resp. $0$) where $\one$ is the function which is identically equal to $1$ in $\Rd$.
For any function $\f:\Rd\to\Rm$, we set $\f^+=\f\vee \bm{0}$ and $\f^-=\f\wedge \bm{0}$.
Throughout the paper we denote by $c$ a positive constant, which may vary from line to line and, if not otherwise specified, may depend at most on $d$, $m$.
We write $c_{\delta}$ when we want to stress that the constant depends on $\delta$.
For any interval $I \subset \R$, we set $\Lambda_I:=\{(t,s)\in I\times I: t> s\}$. Finally, we point out that all the inequalities which involve vector-valued functions are intended componentwise.

\section{Preliminary results}
Let $I$ be either an open right-interval or $I=\R$ and $(\A(t))_{t\in I}$ be a family of second order uniformly elliptic operators defined on smooth vector-valued functions $\boldsymbol{\psi}:\R^d\to\R^m$ by
\begin{align}\label{op_A}
(\A(t)\boldsymbol{\psi})_k(t,x)=
& {\rm Tr}(Q^k(t,x)D^2\psi_k(x))+\langle {\bf b}^k(t,x),\nabla \psi_k(x)\rangle+(C(t,x)\psi(x))_k,\notag\\
=&(\mathcal{A}_k(t)\psi_k)(t,x)+(C(t,x)\psi(x))_k
\end{align}
for any $t\in I$ and $k=1,\ldots, m$.
Fixed $s\in I$, we study the Cauchy problem
\begin{equation}
\left\{
\begin{array}{ll}
D_t\uu=\A(t)\uu, &{\rm in}~(s,+\infty)\times\R^d,\\[1mm]
\uu(s,\cdot)=\f, &{\rm in}~\Rd.
\end{array}
\right.
\label{pb_Cauchy}
\end{equation}
for initial data which are vector-valued bounded and continuous functions $\f:\Rd\to\Rm$. The standing hypotheses considered in the whole paper are the following.

\begin{hyp}\
\label{hyp-base}
\begin{enumerate}[\rm (i)]
\item
The coefficients $q_{ij}^{k}=q_{ji}^{k}$, $b_j^{k}$ and the entries $c_{hk}$ of the not identically vanishing matrix-valued function $C$ belong to $C^{\alpha/2,\alpha}_{\rm loc}(I\times\Rd)$ for some $\alpha\in (0,1)$ and each $i,j=1,\ldots,d$ and $h,k=1, \ldots, m$;
\item
the infimum $\mu_k^0$ over $I\times\Rd$ of the minimum eigenvalue $\mu_k(t,x)$ of the matrix $Q^k(t,x)=(q_{ij}^k(t,x))$ is positive for any $k=1, \ldots, m$;
\item
there does not exist a nontrivial set $K\subset \{1, \ldots, m\}$ such that the coefficients $c_{ij}$ identically vanish on $I\times\Rd$ for any $i\in K$ and $j\notin K$;
\item
for any $J\subset I$ bounded, there exists a positive function $\bm{\varphi}_J\in C^2(\Rd;\Rm)$, blowing up componentwise as $|x|$ tends to $+\infty$
such that $(\A(t)\bm\varphi_J)(x)\le \lambda_J\bm\varphi_J(x)$ for any $t\in J$, $x\in\Rd$ and some positive constant $\lambda_J$;
\item
the off-diagonal entries of the matrix-valued function $C$ are bounded from below on $\Rd$ and the sum of the elements on each row of $C$ is a bounded from above function on $\Rd$.
\end{enumerate}
\end{hyp}

\begin{rmk}
{\rm Some comments on the set of our assumptions are in order.\\
 Hypotheses \ref{hyp-base}(i) and (ii) are a standard regularity assumption on the coefficients of the operator \eqref{op_A} and a standard uniform ellipticity hypothesis on the diffusion matrices $Q^k$, $k=1,\ldots, m$.\\
We consider weakly-coupled systems of parabolic equations and Hypothesis \ref{hyp-base}(iii) is a condition on the entries of the matrix-valued function $C$ which guarantees that the differential system in \eqref{pb_Cauchy} does not contain subsystems with less than $m$ unknowns.\\
Hypothesis \ref{hyp-base}(iv) is the vector-valued version of the scalar one which requires the existence of a Lyapunov function for the elliptic operator associated to the problem. This is typical request when dealing with parabolic problems with unbounded coefficients since it allows to prove a variant of the classical maximum principle.\\
Also Hypothesis \ref{hyp-base}(v) is finalized to prove a maximum principle when, as in our case, the diffusion coefficients and the drift terms can change from line to line. We point out that the assumptions considered here do not imply that the quadratic form associated to the matrix-valued function $C$ is bounded from above in $\Rd$. Indeed, if
\begin{align*}
C(x)=(|x|+1)\begin{pmatrix}
-4 & 1 & 2 & 1 \\
1 & -3 & 1 & 0 \\
0 & 1 & -1 & 0 \\
0 & 2 & 0 & -2
\end{pmatrix},\qquad\;\, x \in \Rd,
\end{align*}
then condition (v) in Hypothesis \ref{hyp-base} is satisfied. However, the matrix $C(0)+C(0)^*$ has a positive eigenvalue $\gamma$. Thus, if $\xi$ denotes a unit eigenvector associated to $\gamma$, then $\langle C(x)\xi,\xi\rangle=\gamma(|x|+1)$ for any $x\in\Rd$.

On the other hand we can find out matrices whose associated quadratic form is non positive definite on $\Rd$ which do not satisfy Hypothesis \ref{hyp-base}(v). Consider for instance the matrix-valued function $C$ defined by
\begin{align*}
C(x)=(|x|+1)\begin{pmatrix}
-4 & 0 & 2 & 1 \\
0 & -3 & 1 & 0 \\
0 & 1 & -1 & 0 \\
1 & 2 & 0 & -2
\end{pmatrix},\qquad\;\, x \in \Rd,
\end{align*}
and notice that the sum of the terms on the last row is positive.\\
We point out that if $C$ is symmetric, the off-diagonal entries of the matrix-valued function $C$ are nonnegative
and the sum of each row of $C$ is nonpositive then the quadratic form associated to the matrix-valued function $C$
is nonpositive. This is an immediate consequence of the Gershgorin's theorem related to the localization of the spectrum of $C$.}
\end{rmk}

In order to deduce uniqueness of a classical solution to problem \eqref{pb_Cauchy} we prove a variant of the classical
maximum principle which holds under more restrictive assumptions on the entries of the matrix- valued function $C$ and
whose proof is deeply based on the existence of the Lyapunov function in Hypothesis \ref{hyp-base}(iv).

\begin{thm}\label{maxprinc}
Let us assume that Hypotheses $\ref{hyp-base}(i)$-$(iv)$ hold true. Further suppose that the off-diagonal entries of the
matrix-valued function $C$ are nonnegative and the sum of each row of $C$ is nonpositive. Then, for any $T>s\in I$, if
$\uu\in C_b([s,T]\times\R^d;\R^m)\cap C^{1,2}((s,T]\times \R^d;\R^m)$ satisfies
\begin{align*}
\left\{
\begin{array}{ll}
D_t\uu-\A(t)\uu\leq 0, & {\rm in}\,\, (s,T]\times\R^d,\\[1mm]
\uu(s,\cdot )\leq 0, & {\rm in}\,\,\R^d,
\end{array}
\right.
\end{align*}
then $\uu\leq 0$ in $[s,T]\times\Rd$.
\end{thm}

\begin{proof}
For each $n\in \N$ we introduce the vector valued function $\vv_n$ defined by
\begin{eqnarray*}
\vv_n(t,x):=\uu(t,x)-\frac{1}{n}e^{\lambda_0(t-s)}\bm\varphi(x),\qquad\;\,(t,x)\in [s,T]\times\Rd,
\end{eqnarray*}
where $\lambda_0$ is a constant larger than $\lambda_{[s,T]}$ and $\bm\varphi=\bm\varphi_{[s,T]}$.
Note that, for any $t \in (s, T]$ and $k=1, \ldots, m$,
\begin{align}\label{ap_0}
D_t v_{n,k}(t,\cdot)-(\A(t)\vv_n)_k(t,\cdot)=&D_tu_{k}(t,\cdot)-(\A(t)\uu)_{k}(t,\cdot)\notag\\
&+\frac{1}{n}e^{\lambda(t-s)}\big((\A(t)\bm\varphi)_{k}-\lambda\varphi_k\big)<0,
\end{align}
due to Hypotheses \ref{hyp-base}(iii),(v).

Let us prove that $\vv_n(t,x)< 0$ for every $(t,x)\in [s,T]\times\R^d$ and $n\in\N$, or equivalently, that
$E_n=\set{t\in[s,T]\tc\vv_{n}(t,x)< 0\text{ for every }x\in\R^d}=[s,T]$. Note that $E_n\neq\varnothing$ since $\vv_n(s,x)< 0$ for any $x\in\Rd$. Moreover, $E_n$ contains a right-neighborhood of $t=s$. Indeed, by continuity, for any $R>0$ there exists $\delta_R>0$ such that $\vv_n<0$ in $[s,s+\delta_R]\times \overline{B_R}$. Since $\vv_n$ tends to $-\infty$, uniformly with respect to $t\in[s,T]$ as $\abs{x}\ra +\infty$, there exists $R_0>0$ such that $\vv_n$ is negative in $[s,T]\times (\Rd\setminus B_{R_0})$. Thus, $E_n$ contains the interval $[s,s+\delta_{R_0}]$. The previous argument also shows that $E_n$ is an interval.

Denote by $\overline t_n$ the supremum of $E_n$ and assume by contradiction that $\bar{t}_n<T$. By continuity $\vv_{n}(\bar{t}_n,\cdot)\le 0$ in $\Rd$, and by definition of $\bar{t}_n$ there exist $k_n\in\set{1,\ldots,m}$ and $\bar{x}_n\in\R^d$ such that $v_{n,k_n}(\bar{t}_n,\bar{x}_n)=0$. Since $\vv_{n}(t,x)\leq 0$ for every $t\leq \bar{t}_n$ and $x\in \Rd$ it follows that $\bar{x}_n$ is a maximum point for $v_{n,k_n}(\bar{t}_n,\cdot)$ and $D_tv_{n,k_n}(t_n,\overline x_n)\ge 0$. Hence,
\begin{align}\label{ap_1}
D_tv_{n,k_n}(\bar{t}_n,\bar{x}_n)-\sum_{i,j=1}^dq_{ij}^{k_n}D_{ij}v_{n,k_n}(\bar{t}_n,\bar{x}_n)-\sum_{i=1}^d b_i^{k_n}D_iv_{n,k_n}(\bar{t}_n,\bar{x}_n)\geq 0,
\end{align}
and, since $c_{k_n,i}\geq 0$ for every $i\neq k_n$ (see Hypothesis \ref{hyp-base}(iii)),
\begin{align}\label{ap_2}
\sum_{i=1}^mc_{{k_n},i}v_{n,i}(\bar{t}_n,\bar{x}_n)=\sum_{\substack{i=1\\ i\neq k_n}}^mc_{{k_n},i}v_{n,i}(\bar{t}_n,\bar{x}_n)\leq 0.
\end{align}
 Estimates \eqref{ap_1} and \eqref{ap_2} contradict \eqref{ap_0}. Thus we get $\vv_n(t,x)< 0$ for any $(t,x)\in [s,T]\times\R^d$ and $n\in \N$. Consequently, letting $n \to +\infty$, we infer that $\uu(t,x)\leq 0$ for every $(t,x)\in[s,T]\times\R^d$.
\end{proof}

\begin{thm}
\label{bbthe}
Under Hypotheses $\ref{hyp-base}$, for any $\f\in C_b(\Rd;\Rm)$ and $s\in I$, the Cauchy problem  \eqref{pb_Cauchy}
admits a unique solution $\uu$ which belongs to $C_b([s,T]\times\Rd;\Rm)\cap C^{1+\alpha/2,2+\alpha}_{\rm loc}((s,+\infty)\times\Rd;\Rm)$ for any $T>s$ and it satisfies the estimate
\begin{equation}\label{moto}
\|\uu(t,\cdot)\|_\infty \le e^{K(t-s)}\|\f\|_\infty,\qquad\;\, t>s,
\end{equation}
for some positive constant $K$ $($explicitely determined in the proof $)$.
\end{thm}

\begin{proof}
We split the proof into two steps. In the first one we consider the case when the off-diagonal elements of the matrix $C$ are nonnegative and the sum of the elements of each row of $C$ is nonpositive. In the second step we address the general case.

{\emph Step 1.}
To begin with, we prove that, if $\uu$ in $C_b([s,T]\times \Rd)\cap C^{1,2}((s,T)\times \Rd)$ is a solution to problem \eqref{pb_Cauchy}, then it is unique and satisfies the estimate
\begin{align}\label{stimaunif}
|u_{i}(t,x)|\leq \max_{k=1,\ldots,m}\|f_{k}\|_{\infty}
\end{align}
for every $(t,x)\in [s,T]\times\R^d$ and $i=1,\ldots, m$. For this purpose, it suffices to apply Theorem \ref{maxprinc} to the function
\begin{eqnarray*}
\vv:=\uu-\max_{k=1,\ldots,m}\|f_{k}\|_{\infty}\bm{\one}.
\end{eqnarray*}
Indeed, clearly $\vv\in C_b([s,T]\times\R^d;\R^m)\cap C^{1,2}((s,T]\times \R^d;\R^m)$ and
\begin{align*}
\vv(s,x)=\uu(s,x)-\max_{k=1,\ldots,m}\|f_{k}\|_{\infty}\bm{\one}=\f(x)-\max_{k=1,\ldots,m}\|f_{k}\|_{\infty}\bm{\one}\leq 0,
\end{align*}
for any $x\in \Rd$. Moreover,
\begin{align*}
D_tv_{k}-(\A(t)\vv)_{k}=&D_tu_{k}-(\A(t)\uu)_{k}+\max_{k=1,\ldots,m}\|f_{k}\|_{\infty}\sum_{i=1}^mc_{ki}\\
=&\max_{k=1,\ldots,m}\|f_{k}\|_{\infty}\sum_{i=1}^mc_{ki}\leq 0,
\end{align*}
due to the fact that $\sum_{i=1}^mc_{ki}\le 0$ in $(s,T]\times\Rd$.
Hence, Theorem \ref{maxprinc} implies that $\vv\leq 0$ in $[s,T]\times\Rd$ and the claim is so proved. By the arbitrariness of $T>s$ we get uniqueness in $[s,+\infty)\times \Rd$.

To prove the existence part let us consider the unique classical solution $\uu_n$ to the Dirichlet problem
\begin{equation*}\left\{\begin{array}{ll}
D_t\uu_n(t,x)=(\A(t)\uu_n)(t,x) & t>s,\ x\in B_n\\[1mm]
\uu_n(t,x)=0 & t>s,\ x\in\partial B_n\\[1mm]
\uu_n(0,x)=\f(x) &\quad\quad\quad x\in B_n,
\end{array}\right.\end{equation*}
(see \cite{Eid69}). By \cite[Theorem 8.15]{PW67}, $\uu_n$ satisfies \eqref{stimaunif} for any $n\in \N$, i.e.,
\begin{align}\label{exi_1}
\|u_{n,i}\|_\infty\leq \max_{k=1,\ldots,m}\|f_{k}\|_{\infty}
\end{align}
holds true for any $n\in \N$ and $i=1, \ldots, m$. The interior Schauder estimates in Theorem \ref{int:ext} together with estimate \eqref{exi_1} guarantee that the sequence $(\uu_n)$ is bounded in $C^{1+\alpha/2,2+\alpha}(E;\R^m)$ where $E$ is any compact subset of $(s,+\infty)\times\R^d$. Classical arguments involving the Ascoli--Arzel\`a theorem and a diagonal procedure allow us to determine a sequence $(\uu_{n_j})\subset(\uu_n)$ converging in $C^{1,2}(E;\Rm)$ to a function $\uu$ belonging to $C_b((s,+\infty)\times\R^d;\R^m)\cap C^{1+\alpha/2,2+\alpha}_{\text{loc}}((s,+\infty)\times\R^d;\R^m)$. Clearly $\uu$ solves the differential equation in \eqref{pb_Cauchy} and estimate \eqref{stimaunif}. To prove the claim we need to show that $\uu$ is continuous at $t=s$ where equals $\f$. For this purpose, we fix $R\in\N$ and let $\theta_R$ be any smooth function such that $\chi_{B_{R-1}}\leq \theta_R\leq \chi_{B_R}$. For any $j\in\N$ such that $n_j\geq R$ we set $\vv_j=\theta_R\uu_{n_j}$. Note that $\vv_j$ belongs to $C([s,T]\times\ol{B}_R;\R^m)\cap C^{1,2}((s,T]\times B_R;\R^m)$ and satisfies the problem
\begin{eqnarray*}
\eqsys{D_t\vv_j(t,x)-\A(t)\vv_j(t,x)=\mathbf{g}_j(t,x), & (t,x)\in(s,T]\times B_R,\\[1mm]
\vv_j(t,x)={\bf 0}, & (t,x)\in (s,T]\times \partial B_R,\\[1mm]
\vv_j(s,x)=\theta_R(x)\f(x), & x\in B_R,}
\end{eqnarray*}
where $g_{j,k}=-2\langle Q^k\nabla u_{n_j,k},\nabla\theta_R\rangle-u_{n_j,k}\mathcal{A}_k\theta_R$. Since all the hypotheses in Proposition \ref{prop-A1} are satisfied, by using \eqref{i_e} and \eqref{exi_1} we get
\begin{eqnarray*}
\abs{\mathbf{g}_j(t,x)}\leq K_R\pa{1+\frac{1}{\sqrt{t-s}}}\max_{k=1,\ldots,m}\|f_{k}\|_{\infty}
\end{eqnarray*}
for every $(t,x)\in (s,s+1)\times B_{R}$ and any $n_j>R$, where $K_R$ is a positive constant independent of $j$. We can write $\vv_j$ by means of the variation-of-constants formula
\begin{eqnarray*}
\vv_j(t,x)=(\mathbf{G}^{D}_R(t,s)(\theta_R\f))(x)+\int_s^t(\mathbf{G}^{D}_R(t,r)\mathbf{g}_j(r,\cdot))(x)dr\qquad t\in[s,T],\ x\in B_R,
\end{eqnarray*}
where $\mathbf{G}^{D}_R(t,s)$ denotes the evolution operator associated to $\A(t)$ in $C_b(\overline{B_R};\Rm)$ with homogeneous Dirichlet boundary conditions. Recalling that $\vv_j=\uu_{n_j}$ in $B_{R-1}$, we get
\begin{eqnarray*}
\abs{\uu_{n_j}(t,\cdot)-\f}\leq \abs{\mathbf{G}^{D}_R(t,s)(\theta_R\f)-\f}+K'_R\sqrt{t-s}\norm{\f}_{\infty}
\end{eqnarray*}
in $B_{R-1}$ for any $t\in(s,s+1)$,
where $K'_R$ is a positive constant independent of $j$. Now, letting $j$ tend to $+\infty$ and, then, $t$ to $s^+$, we conclude that $\uu$ is continuous on $\{s\}\times B_{R-1}$. The arbitrariness of $R$ yields the claim.

\emph{Step 2.}
Now, we consider the general case and prove the claim by using a perturbation argument. We introduce the $m\times m$ matrix $\overline{C}$ with entries $\overline{c}_{ij}=\inf_{I\times\Rd}c_{ij}$, if $i\neq j$, and $\overline{c}_{ii}= \sup_{I\times \Rd}\sum_{k=1}^m c_{ik}-\sum_{k\neq i}\overline{c}_{ik}$, and note that the Cauchy problem \eqref{pb_Cauchy} can be written as follows:
\begin{equation*}
\left\{
\begin{array}{ll}
D_t\uu=\A_0(t)\uu+\overline{C}\uu, &{\rm in}~(s,+\infty)\times\R^d,\\[1mm]
\uu(s,\cdot)=\f, &{\rm in}~\Rd,
\end{array}
\right.
\end{equation*}
where $\A_0:=\A-\overline{C}$ and the off-diagonal elements of the potential of $\A_0$ are nonnegative, whereas the sum of each row is nonpositive.
The existence part can be obtained arguing as in Step 1. Indeed, observing that for any $n\in\N$, the function $\uu_n$ satisfies the uniform estimate $\|\uu_n(t,\cdot)\|_\infty \le e^{\|\overline{C}\|(t-s)}\|\f\|_\infty$ for any $t>s$,
we can prove that problem \eqref{pb_Cauchy} admits a solution $\uu$
which belongs to $C_b([s,T]\times\Rd;\Rm)\cap C^{1+\alpha/2,2+\alpha}_{\rm loc}((s,+\infty)\times\Rd;\Rm)$ for any $T>s$.
Moreover, \eqref{moto} holds true with $K=\|\overline C\|$.

To prove the uniqueness of the solution, it suffices to point out that any solution $\uu$ to the problem \eqref{pb_Cauchy} which belongs to $C_b([s,T]\times\R^d;\R^m)\cap C^{1+\alpha/2,2+\alpha}_{\rm loc}((s,+\infty)\times\Rd;\Rm)$ for each $T>s$ can be written as follows
\begin{equation}\label{int_rep}
\uu(t,\cdot)= \G_0(t,s)\f+\int_s^t \G_0(t,r)(\overline{C}\uu(r,\cdot)))dr,
\end{equation}
where $\{\G_0(t,s):\, t\ge s \in I\}$ denotes the contractive evolution operator associated to $\A_0$ in $C_b(\Rd;\Rm)$. Formula \eqref{int_rep} and the Gronwall Lemma yield immediately
that $\|\uu(t,\cdot)\|_\infty \le e^{\|\overline{C}\|(t-s)}\|\f\|_\infty$ for every $t>s$, whence uniqueness follows.
\end{proof}

As a consequence of Theorem \ref{bbthe} we can define a family of bounded operators $\{\G(t,s)\}_{t\ge s\in I}$ on $C_b(\Rd;\Rm)$ by setting
$\G(t,s)\f=\uu(t,\cdot)$ for any $t>s\in I$, where $\uu$ is the unique solution to the Cauchy problem \eqref{pb_Cauchy} with $\f\in C_b(\Rd;\Rm)$.

\begin{rmk}\label{Neu}{\rm
We stress that the solution $\uu$ of the problem \eqref{pb_Cauchy} could be also approximated by
 the solution to the Neumann-Cauchy problem
\begin{equation*}\left\{\begin{array}{ll}
D_t\uu_n(t,x)=(\A(t)\uu_n)(t,x) & t>s,\ x\in B_n\\[1mm]
\langle \nabla_x\uu_n(t,x), \nu(x)\rangle=0 & t>s,\ x\in\partial B_n\\[1mm]
\uu_n(0,x)=\f(x) & x\in B_n
\end{array}\right.
\end{equation*}
where $\nu$ is the unit normal exterior vector to $\partial B_n$ which is governed by the Neumann evolution operator $\G^{N}_n(t,s)$. Also in this case the sequence $(\G^{N}_n(\cdot,s)\f)$ converges to $\uu$ in $C^{1,2}(E,\R^m)$ for any compact set $E\subset (s,+\infty)\times \Rd$.}
\end{rmk}

Here, we list some continuity properties of the evolution operator $\G(t,s)$ together with an integral representation formula. The proof of this results can be obtained arguing as in \cite[Proposition 3.2 \& Theorem 3.3]{AALT}.

\begin{thm}\label{teo}
If $(\f_n)$ is a bounded sequence of functions in $C_b(\R^d;\R^m)$ then the following properties hold true:
\begin{enumerate}[\rm (i)]
\item
if $\f_n$ converges pointwise to $\f\in C_b(\R^d;\R^m)$, then
${\bf G}(\cdot,s)\f_n$ converges to ${\bf G}(\cdot,s)\f$ in $C^{1,2}(E)$ for any compact set $E\subset (s,+\infty)\times\Rd$;
\item
if $\f_n$ converges to $\f$ locally uniformly in $\R^d$, then ${\bf G}(\cdot,s)\f_n$ converges to
${\bf G}(\cdot,s)\f$ locally uniformly in $[s,+\infty)\times\Rd$.
\end{enumerate}
Moreover, there exists a family of finite Borel measures $\{p_{ij}(t,s,x,dy):\, t>s\in I, x\in \Rd, i,j=1, \ldots, m\}$ such that
\begin{equation}\label{int_rep_1}
(\G(t,s)\f(x))_k= \sum_{i=1}^m\int_{\Rd}f_i(y)p_{ki}(t,s,x,dy), \qquad\;\, \f\in C_b(\Rd;\Rm).
\end{equation}
Finally, through formula \eqref{int_rep_1} $\G(t,s)$ can be extended to $B_b(\Rd;\Rm)$ with a strong Feller evolution operator.
\end{thm}

Now we are interested in finding conditions which ensure the positivity of the evolution operator $\G(t,s)$ in $C_b(\Rd;\Rm)$ in the sense that, if $\f\in C_b(\Rd;\Rm)$ has all nonnegative components, then the function $\G(t,s)\f$ has nonnegative components as well, for any $t>s$. Weakly coupled operators with the same principal part have been considered in \cite{AAL_Inv} extending the result proved in \cite{Ots88Ont} for operators with bounded coefficients. Similar results can be proved also in the case considered here, where, an additional assumption on the matrix-valued function $C$ guarantees also the strict positivity (with the obvious meaning) of the evolution operator $\G(t,s)$. In what follows, in order to simplify the notation we set $I_{\hat{i}}:=\{j\in \N, 1\le j\le m,\,\, j\neq i\}$.

\begin{hyp}\label{allpos}
The off-diagonal entries of the matrix-valued function $C$ are nonnegative.
\end{hyp}

\begin{prop}\label{pro}
Under Hypotheses $\ref{hyp-base}$ and $\ref{allpos}$, if $\f\in C_b(\Rd;\Rm)$ has all nonnegative components and it has at least a component which does not identically vanish in $\Rd$ then $(\G(t,s)\f)_j>0$ in $\Rd$ for any $t>s$ and $j=1,\ldots, m$. Consequently, for any $i,j=1, \ldots, m$, $t>s\in I$ and $x\in \Rd$, each measure $p_{ij}(t,s,x,\cdot)$ is positive and equivalent to the Lebesgue measure.
\end{prop}

\begin{proof}
We split the proof into three steps.

{\em Step 1.} Here, for each $k=1,\ldots,m$ and $i\in\N$, we introduce the sets $H_k^i$, defined by
\begin{equation*}
\left\{\begin{array}{ll}
H_{k}^0=\{j\in\{1,\ldots,m\}\setminus\{k\}: c_{jk}\not\equiv 0\,\,{\rm in}\,\,I\times\Rd\},\vspace{2mm}\\
H_{k}^i=\{j\in \{1,\ldots,m\}\setminus \{k\}\cup\bigcup_{r=0}^{i-1} H^r_k: \exists l\in H^{i-1}_k\,{\rm s.t.}\, c_{jl}\not\equiv 0\,\,{\rm in}\,\,I\times\Rd\},
\end{array}
\right.
\end{equation*}
and prove that, for each $k$, there exists $m_k<m$ such that $H^i_{k}\neq\varnothing$ $(i=1,\ldots,m_k)$ and
$\{1,\ldots,m\}\setminus\{k\}=\bigcup_{i=0}^{m_k}H_{k}^i$.

Let us fix $k\in\{1,\ldots,m\}$ and suppose, by contradiction, that $H^0_k=\varnothing$. This would imply that $c_{jk}=0$ for any $j\neq k$. Clearly this condition contradicts Hypothesis \ref{hyp-base}(iii), taking $K=\{k\}$. Let us now fix $r>0$ such that $\bigcup_{j=0}^rH_k^j$ is properly contained in the set $\{1,\ldots,m\}\setminus\{k\}$ and prove that $H^{r+1}_k\neq\varnothing$. On the contrary, let us assume that $H^{r+1}_k=\varnothing$. This means that,
for any $i\notin H^0_k\cup\cdots H^r_k\cup\{k\}$ and $\ell\in H^r_k$, $c_{i\ell}$ identically vanishes in $I\times\Rd$.
By the definitions of $H_k^i$, $i=0, \ldots, r$, it follows that $c_{ij}$ identically vanishes in $I\times \Rd$ for any $j\in\{k\}\cup H^0_k\cup\cdots H^{r-1}_k$. Summing up we conclude that $c_{ij}\equiv 0$ in $I\times \Rd$ for any $j\in\{k\}\cup H^0_k\cup\cdots H^{r}_k$ and $i \notin \{k\}\cup H^0_k\cup\cdots H^{r}_k$ contradicting again Hypothesis \ref{hyp-base}(iii), taking $K=\{k\}\cup H^0_k\cup\cdots H^{r}_k$. The second statement now follows immediately.

{\em Step 2.} Here, we prove the first part of the claim.
Let $\f\in C_b(\Rd;\Rm)$ be such that $f_k$ does not identically vanish in $\Rd$ and let us show that $(\G^{{D}}_n(t,s)\f)_j$ is positive in $\Rd$ for any $t>s\in I$  and $j\in\{1,\ldots, m\}$. Then, letting $n$ tends to infinity we get the claim by monotonicity.
Let us consider first the case $j=k$ and let $G^{{D}}_{n,k}(t,s)$ be the evolution operator associated to the operator $\mathcal{A}_k+c_{kk}$ in $C(\overline B_n)$ with homogeneous Dirichlet boundary conditions. Since $G^{{D}}_{n,k}(t,s)$  is irreducible, it is known that $G^{{D}}_{n,k}(t,s)f_k>0$ in $\Rd$ for any $t>s$. Taking into account that $(\G^{{D}}_n(\cdot,s)\f)_j$ is nonnegative in $(s,+\infty)\times B_n$ for any $j\in \{1,\ldots, m\}$ (see \cite[Proposition 2.8]{AAL_Inv} with the obvious changes) and that the off-diagonal entries of $C$ are nonnegative functions, using a scalar maximum principle we deduce that
\begin{equation}\label{ma}(\G^{{D}}_n(t,s)\f)_k)(x)\ge (G^{{D}}_{n,k}(t,s)f_k)(x)>0,\qquad (t,x)\in (s,+\infty)\times B_n.
\end{equation}
Now, we fix $j\in\{1,\ldots, m\}\setminus\{k\}$. Clearly, if $f_j$ does not identically vanish the claim follows immediately arguing as above. Hence, let us assume that $f_j\equiv 0$ in $\Rd$. Since $j$ belongs to $\bigcup_{r=0}^{m}H_k^r$ and $H^i_k\cap H_k^j=\varnothing$ for $i\neq j$, there exists a unique $r\in\{0,\ldots,m_k\}$ such that $j\in H_k^{r}$. Now, if $r=0$ then $c_{jk}$ does not identically vanish in $I\times\Rd$ and,
since $u^n_j:=(\G^{{D}}_n(\cdot,s)\f)_j$ satisfies the equation
$D_t u^n_j= \mathcal{A}_j u^n_j+c_{jj}u^n_j+\sum_{h\neq j}c_{jh}u_h^n$ in $(s,+\infty)\times B_n$, we get
\begin{align}\label{mar}
u^n_j(t,\cdot)&= G^{D}_{n,j}(t,s)f_j+\sum_{i\neq j}\int_s^t G^{D}_{n,j}(t,r)(c_{ji}(r,\cdot)u_i^n(r,\cdot))dr\notag\\
&=\sum_{i\neq j}\int_s^t G^{D}_{n,j}(t,r)(c_{ji}(r,\cdot)u_i^n(r,\cdot))dr\notag\\
& \ge \int_s^t G^{D}_{n,k}(t,r)(c_{jk}(r,\cdot)u_k^n(r,\cdot))dr
\end{align}
and the last side of \eqref{mar} is strictly positive in $\Rd$ for any $t>s\in I$. Otherwise if $r>0$, then by definition of $H_k^{r}$, we deduce that there exists $\ell_1\in H_k^{r-1}$ such that $c_{j\ell_1}$ does not identically vanish in $I\times \Rd$. Iterating this argument we conclude that for any $h\le r$ there exist $\ell_h\in H_k^{r-h}$ such that $c_{\ell_{h-1}\ell_h}$ does not identically vanish in $I\times \Rd$. In particular, since $\ell_{r}\in H^0_k$,  $c_{\ell_{r-1}\ell_{r}}\not\equiv 0$ in $I\times \Rd$ and, consequently $c_{\ell_rk}$, does not identically vanish in $I\times \Rd$. The above arguments imply that $(\G^{{D}}_n(\cdot,s)\f)_{\ell_r}$ is positive in $(s,+\infty)\times \Rd$. But, again, since $c_{\ell_{r-1}\ell_r}\not\equiv 0$ in $I\times \Rd$ we get that $(\G^{{D}}_n(\cdot,s)\f)_{\ell_{r-1}}$ is positive in $(s,+\infty)\times \Rd$. Iterating this procedure we finally conclude that $(\G^{{D}}_n(\cdot,s)\f)_{j}$ is positive in $(s,+\infty)\times \Rd$.

 As a byproduct we deduce that for any $t>s$, $x\in \Rd$ and $i,j=1, \ldots, m$ the measure $p_{ij}(t,s,x,dy)$ is positive. Indeed, $p_{ij}(t,s,x,\Rd)= (\G(t,s)\e_j)_i(x)>0$.\\
{\em Step 3.} Here we prove that the measures $\{p_{ij}(t,s,x,dy): t>s,\, x\in \Rd,\, i,j=1, \ldots, m\}$ are equivalent to the Lebesgue measure.
Arguing as in \cite[Theorem 3.3]{AALT} it can be proved that if $A$ is a Borel set with null Lebesgue measure then $\G(t,s)(\chi_A \e_j)(x)=\bf{0}$ for any $t>s$, $x\in \Rd$ and $j=1, \ldots, m$. Consequently, since
\begin{equation}\label{ker}
p_{ij}(t,s,x,A)= (\G(t,s)(\chi_A \e_j))_i(x),
\end{equation}
each $p_{ij}(t,s,x,dy)$ is absolutely continuous with respect to the Lebesgue measure.
On the other hand, let us assume that $p_{ij}(t,s,x,A)=0$ for any $i,j,t,s$ and $x$ as above and prove that the Lebesgue measure of $A$ is zero. Suppose, by contradiction, that this measure is positive. Then, the strong Feller property of $\G^{D}_{n}(t,s)$ and $G^{D}_{n,k}(t,s)$ allows to extend estimate \eqref{ma} to any bounded Borel function. In particular $(\G^{D}_{n}(t,s)\chi_A \e_j)_j \ge G^{D}_{n,j}(t,s)\chi_A$ for any $t>s$ and $j=1,\ldots, m$.
Letting $n \to +\infty$ we infer that $(\G(t,s)\chi_A \e_j)_j \ge G_j(t,s)\chi_A>0$ for any $t>s$.
The vector-valued function $\G(t,s)(\chi_A \e_j)$ is the unique solution to the Cauchy problem
\begin{equation*}
\left\{ \begin{array}{ll}
D_t \uu= \A(t) \uu,\qquad\; &(s+\varepsilon, +\infty)\times \Rd,\\
\uu(s+\varepsilon,\cdot)= \G(s+\varepsilon,s)(\chi_A \e_j),\qquad\; &\Rd
\end{array}
\right.
\end{equation*}
for any $\varepsilon>0$. Thus, since $\G(s+\varepsilon,s)(\chi_A \e_j)$ is a bounded, continuous, nonnegative and not identically vanishing function, by the first part of the proof we conclude that $(\G(t,s)(\chi_A \e_j))_i$ is positive for any $t>s$ and $i=1,\ldots, m$ contradicting formula \eqref{ker}.
\end{proof}

%

\section{Compactness of $\G(t,s)$ in the space of continuous functions}\label{comp-sect}

In this section we prove some compactness results for the evolution operator $\G(t,s)$ in the space of continuous and bounded functions. The main results are stated in Theorems \ref{thm_comp1} and \ref{thm_comp_2}. More precisely, the first theorem provides us with sufficient conditions for the evolution operator $\G(t,s)$ to be \emph{locally compact} in $C_b(\Rd;\Rm)$ uniformly with respect to $t>s\in I$, in the sense that for any $s\in I$ and $(\f_n)_n\subset C_b(\Rd;\Rm)$, the sequence $(\G(\cdot,s)\f_n)_n$ admits a subsequence which converges uniformly in $(t_0,+\infty)\times B_k$ for any $k >0$ and some $t_0\ge s\in I$.
The second result is concerned with the compactness of the evolution operator $\G(t,s)$ in $C_b(\Rd;\Rm)$ for $(t,s)\in \Lambda_J$ and bounded $J\subset I$. To prove these results we need to straighten the hypotheses on the coefficients of the operator \eqref{op_A}.

\begin{hyp}\label{Lya2}
\begin{enumerate}[\rm (i)]
\item
For any bounded interval $J\subset I$ there exist $m$-nonnegative functions $\psi_k^J\in C^2(\Rd)$ $(k=1, \ldots, m)$, blowing up as $|x|\to +\infty$, a real constant $\delta_J>0$
such that
\begin{eqnarray*}
(\mathcal{A}_k(t)\psi_k^J)(x) \le \delta_J\psi_k^J(x),\qquad\;\, t\in J,\;\, x\in \Rd,\, k=1, \ldots, m;
\end{eqnarray*}
\item
the sum of the elements of each row of the matrix-valued function $C$ is nonpositive in $\Rd$.
\end{enumerate}
\end{hyp}

\begin{lemm}\label{intbypa}
Under Hypotheses $\ref{hyp-base}(i)$-$(iii)$, $\ref{allpos}$ and $\ref{Lya2}$, for any $x\in\Rd$ and $\f \in C^2_b(\Rd;\Rm)$ constant and nonnegative outside a ball, the function $(\G(t,\cdot)\bm\A(\cdot)\f)(x)$ is locally integrable in $I\cap (-\infty, t]$ and
\begin{equation}\label{isem}
(\G(t,s_1)\f)(x)-(\G(t,s_0)\f)(x)\ge -\int_{s_0}^{s_1}(\G(t,\sigma)\bm\A(\sigma)\f)(x)d\sigma
\end{equation}
for any $s_0\le s_1 \le t$ and $x\in \Rd$.
\end{lemm}

\begin{proof}
First of all, we show that
\begin{equation}\label{suppcomp}
(\G(t,s_1)\f)(x)- (\G(t,s_2)\f)(x)=-\int_{s_1}^{s_2}(\G(t,\sigma)\bm\A(\sigma)\f)(x)d\sigma
\end{equation}
for any $\f \in C^2_c(\Rd;\Rm)$. To this aim, let us consider the evolution operator $\G^D_n(t,s)$ associated to $\bm \A$ in $C_b(B_n;\Rm)$ with homogeneous Dirichlet boundary conditions.
It is well known that, for any $\f\in C^2_c(\Rd;\Rm)$ and $n$ sufficiently large such that ${\rm supp}(f_i)\subset B_n$ for any $i=1, \ldots, m$, it holds that
\begin{eqnarray*}
(\G^D_n(t,s_1)\f)(x)- (\G^D_n(t,s_2)\f)(x)=-\int_{s_1}^{s_2}(\G^D_n(t,\sigma)\bm\A(\sigma)\f)(x)d\sigma
\end{eqnarray*}
for any $s_0\le s_1 \le t$ and $x\in \Rd$. Since the function $\bm\A(\sigma)\f\in C_b(\Rd;\Rm)$ for any $\sigma \in [s_1,s_2]$, using the approximation arguments in the proof of Theorem \ref{bbthe}, we can let $n$ tend to $+\infty$ and deduce \eqref{suppcomp}, by the dominated convergence theorem.

Now, let $\f$ be as in the statement. Thanks to \eqref{suppcomp} and to the linearity of $\G(t,s)$, we can limit ourselves to proving \eqref{isem} for $\f=\bm{\one}$.
First, assume that all the entries of the matrix-valued function $C$ are bounded in $J\times \Rd$ for any bounded $J\subset I$. In this case, since $\bm{\one}$ belongs to the domain of the generator of the evolution operator $\G_n^{N}(t,s)$ associated to $\bm \A$ in $C_b(B_n;\Rm)$ with homogeneous Neumann boundary conditions, it follows that
\begin{eqnarray*}
(\G^{N}_n(t,s_1)\bm{\one})(x)- (\G^{N}_n(t,s_2)\bm{\one})(x)=-\int_{s_1}^{s_2}(\G^{ N}_n(t,\sigma)(C(\sigma,\cdot)\bm{\one}))(x)d\sigma.
\end{eqnarray*}
By Remark \ref{Neu}, estimate \eqref{stimaunif} and the dominated convergence theorem we get
\begin{equation*}
(\G(t,s_1)\bm{\one})(x)- (\G(t,s_2)\bm{\one})(x)=-\int_{s_1}^{s_2}(\G(t,\sigma)(C(\sigma,\cdot)\bm{\one}))(x)d\sigma.
\end{equation*}

Finally, if the matrix-valued function $C$ is unbounded, we can consider a sequence of functions $\vartheta_n\in C_c(\Rd)$ such that $\chi_{B_n}\le \vartheta_n\le \chi_{B_{n+1}}$ for any $n \in \N$, and set $C_n=\vartheta_n C$ for any $n\in \N$. Clearly, thanks to Hypothesis \ref{Lya2}, for any $n\in \N$ the operator $\bm\A_n(t)=\A(t)-C(t,\cdot)+C_n(t,\cdot)$ satisfies Hypotheses \ref{hyp-base}. Thus, we can consider the positive evolution operator $\G_n(t,s)$ associated with $\bm\A_n$ in $C_b(\Rd;\Rm)$.
Since $C_m\in C(I;C_c(\Rd;\Rm))$ and, by Hypothesis \ref{Lya2}(ii), $C_m\bm{\one}\le C_n\bm{\one}$ for any $m>n$ we can estimate
\begin{align}\label{est}
(\G_m(t,s_1)\bm{\one})(x)- (\G_m(t,s_2)\bm{\one})(x)=& -\int_{s_1}^{s_2}(\G_m(t,\sigma)(C_m(\sigma,\cdot)\bm{\one}))(x)d\sigma\notag\\
\ge & -\int_{s_1}^{s_2}(\G_m(t,\sigma)(C_n(\sigma,\cdot)\bm{\one}))(x)d\sigma
\end{align}
for any $m>n$, $m\in \N$. We now observe that $\G_m(t,s)\f$ converges to $\G(t,s)\f$ pointwise in $\Rd$, for any $I\in s<t$, as $m\to +\infty$ for any $\f\in C_b(\Rd,\Rm)$.
Indeed, the Schauder estimates in Theorem \ref{int:ext} show that there exists a subsequence $(m_k)$ such that $\G_{m_k}(\cdot,s)\f$ converges to a function ${\bm v}\in C^{1,2}((s,+\infty)\times\Rd;\R^m)$. Function ${\bm v}$ is bounded since each $\G_m(\cdot,s)\f$ is bounded in $(s,+\infty)\times\Rd$.
To identify ${\bm v}$ with $\G(\cdot,s)\f$, we need to show that $\bm{v}$ can be extended by continuity on $\{s\}\times\Rd$, where it equals $\f$.
For this purpose, we start considering $\f\in C^2_c(\Rd;\Rm)$ and note that formula \eqref{suppcomp} holds true with the evolution operator $\G(t,s)$ being replaced by $\G_m(t,s)$. From that formula it is clear that
\begin{eqnarray*}
\|\G_{m_k}(t,s)\f-\f\|_{\infty}\le c(t-s)\|\f\|_{\infty},\qquad\;\,t>s.
\end{eqnarray*}
Letting $k$ tend to $+\infty$, the continuity of ${\bm v}$ at $t=s$ follows at once. The above arguments also show that from any subsequence of $(\G_m(\cdot,s)\f)$ we can extract a subsequence which converges (locally uniformly on $(s,+\infty)\times\Rd$) to $\G(\cdot,s)\f$. Thus, all the sequence $(\G_m(\cdot,s)\f)$ converges to $\G(\cdot,s)\f$ as $m\to +\infty$. A density argument shows that $\vv$ is continuous on $\{s\}\times\Rd$, where it equals $\f$, also when $\f$ is continuous in $\Rd$ with compact support. Moreover,
all the sequence $(\G_m(\cdot,s)\f)$ converges to $\G(\cdot,s)\f$ as $m\to +\infty$.
For a general $\f\in C_b(\Rd)$, we fix $M>0$ and a smooth function $\vartheta$ such that $\chi_{B_M}\le\vartheta\le\chi_{B_{2M}}$. We split $\G_m(t,s)\f=\G_m(t,s)(\vartheta\f)+\G_m(t,s)((1-\vartheta)\f)$. Since $\G_m(t,s)$ is a positive evolution operator and $|(1-\vartheta)\f|\le (1-\vartheta)\|\f\|_{\infty}\1$ componentwise, we can estimate
\begin{align*}
|\G_m(t,s)((1-\vartheta)\f)|\le &\|\f\|_{\infty}\G_m(t,s)((1-\vartheta)\1)=\|\f\|_{\infty}[\G_m(t,s)\1-\G_m(t,s)(\vartheta\1)]\\
\le &\|\f\|_{\infty}[\1-\G_m(t,s)(\vartheta\1)],
\end{align*}
where we have used Theorem \ref{maxprinc} to derive the last inequality.
Thus,
\begin{eqnarray*}
|\G_{m_k}(t,s)\f-\f|\le |\G_{m_k}(t,s)(\vartheta\f)-\f|+\|\f\|_{\infty}[\1-\G_{m_k}(t,s)(\vartheta\1)].
\end{eqnarray*}
Letting $k$ tend to $+\infty$, we obtain
\begin{eqnarray*}
|\vv(t,\cdot)-\f|\le |\G_{m_k}(t,s)(\vartheta\f)-\f|+\|\f\|_{\infty}[\1-\G(t,s)(\vartheta\1)].
\end{eqnarray*}
From this inequality, it follows that $\vv$ tends to $\f$ as $t\to s^+$, uniformly with respect to $x\in B_M$. The arbitrariness of $M>0$ allows us to conclude that
$\vv=\G(\cdot,s)\f$ as claimed.

Now, we can let $m$ tend to $+\infty$ in \eqref{est} and get
\begin{align*}
(\G(t,s_1)\bm{\one})(x)- (\G(t,s_2)\bm{\one})(x)\ge & -\int_{s_1}^{s_2}(\G(t,\sigma)(C_n(\sigma,\cdot)\bm{\one}))(x)d\sigma.
\end{align*}
Since $\G(t,s)$ is a positive operator and the sequence $(C_n\bm{\one})$ is decreasing componentwise, we can apply twice the monotone convergence theorem to pass to the limit as $n \to +\infty$ and get
\begin{align*}
(\G(t,s_1)\bm{\one})(x)- (\G(t,s_2)\bm{\one})(x)\geq & -\int_{s_1}^{s_2}(\G(t,\sigma)(C(\sigma,\cdot)\bm{\one}))(x)d\sigma.
\end{align*}
The proof is complete.
\end{proof}

\begin{hyp}\label{Lya3}
There exist a nonnegative function $\varphi\in C^2(\Rd)$, blowing up as $|x|\to +\infty$, constants $a,c>0$ and $t_0\in I$ such that
\begin{eqnarray*}
(\bm \A(t) (\varphi\bm{\one}))(x) \le (a -c\varphi(x))\bm{\one},\qquad\;\, t\ge t_0,\, x\in \Rd.
\end{eqnarray*}
\end{hyp}

\begin{rmk} {\rm Note that under Hypothesis \ref{Lya2}(ii), Hypotheses \ref{Lya2}(i) and \ref{Lya3} are both satisfied if there exists a nonnegative function $\varphi\in C^2(\Rd)$, blowing up as $|x|\to +\infty$ and constants $a,c>0$, $t_0\in I$ such that
$(\mathcal A_i(t)\varphi)(x) \le a -c\varphi(x)$ for any $t\ge t_0\in I$, $x\in \Rd$ and $i=1,\ldots,m$.}
\end{rmk}

\begin{lemm}\label{pre_lem}
Let the assumptions of Lemma $\ref{intbypa}$ and Hypothesis $\ref{Lya3}$ be satisfied. Then, the function $\G(t,s)(\varphi\bm{\one})$ is well defined for any $t_0\le s\le t\in I$. Moreover, for any fixed $x \in \Rd$, the function $(t,s)\mapsto (\G(t,s)(\varphi\bm{\one}))(x)$ is bounded in $\Lambda_0=\{(t,s)\in I\times I: t_0\le s\le t\}$ and satisfies
the inequality $(\G(t,s)(\varphi\bm{\one}))(x)\le ((\varphi+ac^{-1})\bm{\one})(x)$ for any $x\in \Rd$ and $(t,s)\in \Lambda_0$.
\end{lemm}
\begin{proof}
First we prove that the function $\G(t,s)(\varphi\bm{\one})$ is well defined in $\Rd$ for any $t>s\ge t_0$.
To this aim, for any $n\in \N$ choose $\psi_n\in C^2([0,+\infty))$ such that
\begin{enumerate}[\rm (i)]
\item $\psi_n(x)=x$ for $x\in [0,n]$;
\item $\psi_n(x)=n+1/2$ for $x \ge n+1$;
\item $0\le \psi'_n\le 1$ and $\psi''_n \le 0$.
\end{enumerate}
Note that the previous conditions imply that  $\psi'_n(x)x\le \psi_n(x)$ for any
$x\in [0,+\infty)$. Moreover, since the functions $\varphi_n= \psi_n\circ \varphi$ belong to
$C^2_b(\Rd)$ and are constant outside a compact set, Lemma \ref{intbypa} and the nonnegativity of $\G(t,s)$ yield
\begin{align*}
\varphi_n(x)&\ge \varphi_n(x)-(\G(t,s)\varphi_n\bm{\one})_i(x)\notag\\
& \ge -\int_{s}^{t}(\G(t,\sigma)\bm\A(\sigma)\varphi_n\bm{\one})_i(x)d\sigma\\
& = -\sum_{j=1}^m\int_{s}^{t}\int_{\Rd}(\bm\A(\sigma)\varphi_n\bm{\one})_j(y)p_{ij}(t,\sigma,x,dy)d\sigma\\
& = -\sum_{j=1}^m\int_{s}^{t}\int_{\Rd}\psi'_n(\varphi(y))(\mathcal{A}_j(\sigma)\varphi)(y)p_{ij}(t,\sigma,x,dy)d\sigma\\
&\quad -\sum_{j=1}^m\int_{s}^{t}\int_{\Rd}\psi''_n(\varphi(y))\langle Q^j(\sigma,y)\nabla \varphi(y), \nabla \varphi(y)\rangle p_{ij}(t,\sigma,x,dy)d\sigma\\
& \quad -\sum_{j,k=1}^m\int_{s}^{t}\int_{\Rd}\psi_n(\varphi(y))c_{jk}(\sigma, y)p_{ij}(t,\sigma,x,dy)d\sigma
\end{align*}
for any $i=1, \ldots, m$, $t>s\in I$ and $x \in \Rd$, where $\mathcal A_j(\sigma)$ is defined in \eqref{op_A}. \\
Using Hypothesis \ref{hyp-base}(ii) and recalling that $\mathcal{A}_j(\sigma)\varphi= ({\bm \A}(\sigma)(\varphi\bm{\one}))_j-(C(\sigma,\cdot)\varphi\bm{\one})_j$ for any $j=1, \ldots, m$, we estimate
\begin{align}\label{caldo}
&\varphi_n(x)-(\G(t,s)\varphi_n\bm{\one})_i(x)\notag\\
\ge &-\sum_{j=1}^m\int_{s}^{t}\int_{\Rd}\psi'_n(\varphi(y))(\bm\A(\sigma)\varphi\bm{\one})_j(y)p_{ij}(t,\sigma,x,dy)d\sigma\notag\\
& -\sum_{j=1}^m\int_{s}^{t}\int_{\Rd}\left[\psi_n(\varphi(y))-\psi'_n(\varphi(y))\varphi(y)\right]\sum_{k=1}^mc_{jk}(\sigma,y)p_{ij}(t,\sigma,x,dy)d\sigma\notag\\
\ge  &-\sum_{j=1}^m\int_{s}^{t}\int_{\Rd}\psi'_n(\varphi(y))(\bm\A(\sigma)\varphi\bm{\one})_j(y)p_{ij}(t,\sigma,x,dy)d\sigma,
\end{align}
where in the last line we have used Hypothesis \ref{Lya2}(ii).
Now, we can split
\begin{align*}
&-\sum_{j=1}^m\int_{s}^{t}\int_{\Rd} \psi'_n(\varphi(y))(\bm\A(\sigma)\varphi\bm{\one})_j(y)p_{ij}(t,\sigma,x,dy)d\sigma\\
=& \sum_{j=1}^m\int_{s}^{t}\int_{\Rd}\psi'_n(\varphi(y))\left[a-(\bm\A(\sigma)\varphi\bm{\one})_j(y)\right]p_{ij}(t,\sigma,x,dy)d\sigma\\
& - a\sum_{j=1}^m\int_{s}^{t}\int_{\Rd}\psi'_n(\varphi(y))p_{ij}(t,\sigma,x,dy)d\sigma,
\end{align*}
where $a$ is the constant in Hypothesis \ref{Lya3}. The monotonicity of the sequence $(\psi_n'(x))$ for any $x\in\Rd$ and the monotone convergence theorem yield immediately that both integrals in the right-hand side of the previous formula converge. Thus, since $\varphi_n(x)$ converges to $\varphi(x)$ as $n\to +\infty$ for any $x \in \Rd$, taking the limit as $n \to +\infty$ in \eqref{caldo}, it follows that
$(\G(t,s)\varphi\bm{\one})(x)$ is well defined for any $t\ge s\in \Lambda$, $x\in \Rd$ and
\begin{align*}
(\G(t,s)\varphi\bm{\one})_i(x) &\le \varphi(x)+ \int_{s}^{t}(\G(t,\sigma)(\bm\A(\sigma)\varphi\bm{\one}))_i(x)d\sigma\\
& \le\varphi(x)+ \int_{s}^{t}(a-c(\G(t,\sigma)(\varphi\bm{\one}))_i(x))d\sigma\\
&\le \varphi(x)+a(t-s)
\end{align*}
for any $i=1,\ldots, m$ and $(t,s)\in \Lambda_0$, where we used the fact that $\G(t,\sigma)\1\le\1$.

To complete the proof, for any $i=1, \ldots, m$, $t>s\ge t_0$ and $x \in \Rd$ we define $g_i(s)= (\G(t,s)\varphi\bm{\one})_i(x)$. Arguing as above it can be proved that
\begin{eqnarray*}
g_i(s)-g_i(r)\le \int_r^s (a-c g_i(\sigma))d\sigma,\qquad t_0\le r \le s\le t.
\end{eqnarray*}
From this inequality it follows easily that the function $\zeta:[s,t]\to\R$, defined by
\begin{eqnarray*}
\zeta(r)=\bigg (g_i(s)-\frac{a}{c}+\int_s^r(cg_i(\sigma)-a)d\sigma\bigg )e^{-cr},\qquad\;\,r\in [s,t],
\end{eqnarray*}
is weakly differentiable and its derivative is almost everywhere nonnegative in $[s,t]$. This implies that
$\zeta(s)\le\zeta(t)$, which is the claim.
\end{proof}

\begin{rmk}
{\rm In the proof of the previous lemma, Hypothesis \ref{Lya2}(ii) has played a crucial role. It is for this reason that we needed to consider a vector-valued Lyapunov function with all the components equal each other.}
\end{rmk}

\begin{coro}\label{tight-uni}
Under the hypotheses of Lemma $\ref{pre_lem}$,
$\sup_{t>s}p_{ij}(t,s,x,\Rd\setminus B_r)$ converges to $0$, for any $i,j=1, \ldots, m$ and $s\ge t_0$ $($where $t_0$ is defined in Hypothesis $\ref{Lya3})$, as $r\to +\infty$, locally uniformly with respect to $x \in \Rd$.
\end{coro}
\begin{proof}
The proof of this result is quite standard. However for the sake of completeness we provide a sketch of it.
Taking into account the positivity of the transition kernels, it holds that
\begin{align}\label{loc_comp}
p_{ij}(t,s,x,\Rd\setminus B_r)=& \int_{\Rd\setminus B_r} p_{ij}(t,s,x,dy)\le \frac{1}{\inf_{\Rd\setminus B_r}\varphi}\int_{\Rd\setminus B_r}\varphi p_{ij}(t,s,x,dy)\notag\\
\le &  \frac{1}{\inf_{\Rd\setminus B_r}\varphi}(\G(t,s)\varphi \bm{\one})_i(x)\le \frac{1}{\inf_{\Rd\setminus B_r}\varphi}(\varphi(x)+a c^{-1})
\end{align}
for any $i,j=1, \ldots, m$.
The claim follows since $\varphi$ blows up as $|x|\to +\infty$.
\end{proof}

Now we prove the first compactness result for the evolution operator $\G(t,s)$. Note that this result improves that in Theorem \ref{teo}(ii).
Indeed here we gain an uniform convergence in time of $\G(\cdot, s)\f_n$ to $\G(\cdot,s)\f$ as $n\to +\infty$ when $(\f_n)$ is a sequence approaching $\f$ locally uniformly in $\Rd$.

\begin{thm}\label{thm_comp1}
Assume that Hypotheses $\ref{hyp-base}(i)$-$(iii)$, $\ref{allpos}$, $\ref{Lya2}$ and $\ref{Lya3}$ hold true and let $(\f_n)\subset C_b(\Rd;\Rm)$ be a bounded sequence converging locally uniformly in $\Rd$ to $\f$, as $n\to+\infty$. Then, for any $s\ge t_0$ $($where $t_0$ is defined in Hypothesis $\ref{Lya3})$ $\G(\cdot,s)\f_n$ converges uniformly to $\G(\cdot,s)\f$ in $(s,+\infty)\times B_r$ for any $r>0$, as $n \to +\infty$. In general, for any sequence $(\f_n)\subset C_b(\Rd;\R^m)$, there exists a subsequence $(\f_{n_k})$ such that $\G(\cdot,s)\f_{n_k}$ converges uniformly in $(t_0,+\infty)\times B_r$ for every $r>0$.
\end{thm}
\begin{proof}
Let $(\f_n)$ be a sequence as in the first part of the statement and assume that $\sup_{n\in\N}\|\f_n\|_{\infty}\le M$. Let $t>s\ge t_0$ and $x \in B_k$ for some $k\in\N$. Then, for any $i=1, \ldots, m$ we can estimate
\begin{align}\label{condi}
|(\G(t,s)(\f_n-\f))_i(x)|&\le \sum_{j=1}^m\int_{B_r}|f_{n,j}(y)-f_j(y)|p_{ij}(t,s,x,dy)\notag\\
&+\sum_{j=1}^m\int_{\Rd\setminus B_r}|f_{n,j}(y)-f_j(y)|p_{ij}(t,s,x,dy)\notag\\
& \le \|\f_n-\f\|_{C_b(B_r;\Rd)}\sum_{j=1}^mp_{ij}(t,s,x,B_r)\notag\\
&+2M\sum_{j=1}^m\sup_{t>s}\sup_{x\in B_k}p_{ij}(t,s,x,\Rd\setminus B_r)
\end{align}
for every $r>0$ and $n\in\N$. Since $\sum_{j=1}^mp_{ij}(t,s,\cdot,B_r)=(\G(t,s)\chi_{B_r}\bm{\one})_i$, by estimate \eqref{stimaunif} it follows that
$\sup_{x\in \Rd}\sum_{j=1}^mp_{ij}(t,s,x,B_r)\le 1$ for any $t>s$ and $r>0$. Thus, letting $n$ tend to $+\infty$ in \eqref{condi} we obtain that
\begin{eqnarray*}
\limsup_{n\to +\infty}\|(\G(\cdot,s)(\f_n-\f))_i\|_{C_b((s,+\infty)\times B_k;\Rm)}\le 2M\sum_{j=1}^m\sup_{t>s}\sup_{x\in B_k}p_{ij}(t,s,x,\Rd\setminus B_r)
\end{eqnarray*}
for every $r>0$. Finally, letting $r$ tend to $+\infty$ and using Corollary \ref{tight-uni} we conclude that
\begin{eqnarray*}
\limsup_{n\to +\infty}\|(\G(\cdot,s)(\f_n-\f))_i\|_{C_b((s,+\infty)\times B_k;\Rm)}\le 0
\end{eqnarray*}
 and the first part of the claim is so proved.

To conclude, let us consider a sequence $(\f_n)\subset C_b(\Rd;\R^m)$ for any $n\in \N$ and $r\in I$. The Schauder estimates \eqref{Schauder} and estimate \eqref{stimaunif} yield that, for any fixed  $t_0>s$, the sequence $(\G(t_0,s)\f_n)$
is bounded in $C^{2+\alpha}(B_r;\Rm)$ for any $r>0$. Then, up to subsequences, it converges locally uniformly in $\Rd$ to some function $\g\in C_b(\Rd;\Rm)$.  Thus,
since $|\G(t,s)\f_{n_k}-\G(t,t_0)\g|=|\G(t,t_0)\left(\G(t_0,s)\f_{n_k}-\g\right)|$
in $\Rd$ for every $t>t_0>s$ and $k\in\N$, applying the first part of the claim to the sequence $(\G(t_0,s)\f_{n_k}-\g)_k$ we conclude the proof.
\end{proof}

Now, we are interested in finding conditions that ensure that, for any bounded interval $J\subset I$ and any fixed $(t,s)\in \Lambda_J$ the operator $\G(t,s)$ is compact in $C_b(\Rd;\Rm)$. First of all, let observe that the compactness of $\G(t,s)$ in $C_b(\Rd;\Rm)$ is equivalent to the tightness of the measures $\{p_{ij}(t,s,x,\cdot): x\in \Rd\}$, $i,j=1, \ldots,m$ (see formula \eqref{int_rep_1}), as the next proposition states.

\begin{prop}\label{equi-tight}
Let $J\subset I$ be a bounded interval and $(t,s)\in \Lambda_J$. The evolution operator $\G(t,s)$ is compact in $C_b(\Rd;\Rm)$  if and only if the measures $\{p_{ij}(t,s,x, \cdot): x\in \Rd\}$ are tight for any $i,j=1, \ldots, m$, i.e., for any $\varepsilon>0$ there exists $r>0$ such that  $\sup_{x \in \Rd}p_{ij}(t,s,x,\Rd\setminus B_r)<\varepsilon$ for any $i,j=1, \ldots, m$.
\end{prop}
\begin{proof}
The proof follows adapting the arguments in \cite[Theorem 4.1 ]{AALT}, recalling that the measures $p_{ij}(t,s,x,\cdot)$ are nonnegative for any $t>s\in I$, $x\in \Rd$ and $i,j=1, \ldots,m$.
\end{proof}

Differently from the case considered in \cite{AALT} where a domination of $\G(t,s)$ in terms of a scalar semigroup reduces the problem of finding conditions that ensure the tightness of the measures $p_{ij}(t,s,x,\cdot)$ to the same problem for the kernel associated to the scalar semigroup in $C_b(\Rd)$, here we argue directly with the vector valued operator $\G(t,s)$. To this aim we need to strengthen Hypothesis \ref{Lya3} as follows.

\begin{hyp}\label{hyp_comp_2}
There exist $R>0$,  $I\ni d_1<d_2$ and
\begin{enumerate}[\rm(i)]
\item a positive function $\varphi \in C^2(\Rd)$, blowing up as $|x|\to +\infty$, and $m$-convex functions
$h_i:[0,+\infty)\to\R$, $i=1, \ldots, m$, with $1/h_i\in L^1((M,+\infty))$ for some positive $M$ such that $(\A(t)\varphi\bm{\one})_i(x)\le -h_i(\varphi(x))$ for any
$t\in [d_1,d_2]$, $x\in \Rd\setminus B_R$ and $i=1, \ldots, m$;
\item
bounded functions $w_k\in C^2(\Rd\setminus B_R)$ $(k=1,\ldots, m)$, with $\inf_{x\in \Rd\setminus B_R}w_k(x)>0$ such that $((\mathcal{A}_k(t)+c_{kk}(t,\cdot))w_k)(x)-\mu w_k(x)\ge 0$ for any $(t,x)\in [d_1,d_2]\times (\Rd\setminus B_R)$, $k=1,\ldots, m$ and some $\mu \in \R$.
\end{enumerate}
\end{hyp}

\begin{thm}\label{thm_comp_2}
Assume that Hypotheses $\ref{hyp-base}(i)$-$(iii)$ and $(v)$, $\ref{allpos}$ and $\ref{hyp_comp_2}$ hold true. Then $\G(t,s)$ is compact in $C_b(\Rd;\Rm)$ for any $(t,s)\in \Lambda_I$ with $s\le d_2$ and $t\ge d_1$.
\end{thm}
\begin{proof}
Due to its length we divide the proof into three steps.\\
{\em Step 1.} Here, we prove that for any $s_0, t_0 \in [d_1,d_2]$ with $s_0<t_0$, there exists a positive constant $c_{0}$ such that
\begin{equation}\label{bfb}
(\G(t,s)\bm{\one})_k(x)\ge c_0, \qquad\;\,s_0\le s\le t\le t_0,\,\,x \in \Rd,\,\,k=1,\ldots,m.
\end{equation}
Let us fix $s_0, t_0$ as above and observe that, under our assumptions, \cite[Proposition 4.3]{AngLor10Com} can be applied and implies that there exists a positive constant $c_0$
such that $(G_k(t,s)\one)(x)\ge c_0$ for any $s_0\le s\le t\le t_0$, $x \in \Rd$ and $k=1,\ldots,m$. Here, $G_k(t,s)$ denotes the positive evolution operator associated to $\mathcal{A}_k(t)+c_{kk}(t,\cdot)$ in $C_b(\Rd;\Rm)$.
In order to prove \eqref{bfb} it suffices to prove that $(\G(t,s)\bm{\one})_k\ge G_k(t,s)\one$ for any $k=1, \ldots, m$ and $t\ge s \in I$. For this purpose we observe that, for any non positive initial datum $\f\in C_b(\Rd;\Rm)$, the function $w_k(t,x)=(\G(t,s)\f)_k- G_k(t,s)f_k$ vanishes at $t=s$ and satisfies the inequality
\begin{eqnarray*}
D_t w_k(t,\cdot)-(\mathcal{A}_k(t)+c_{kk}(t,\cdot))w_k(t,\cdot)= \sum_{i\neq k}c_{ki}(\G(t,s)\f)_i\le 0
\end{eqnarray*}
for any $t>s \in I$, where in the last inequality we have used the positivity of $\G(t,s)$ and Hypothesis \ref{allpos}. Thanks to Hypothesis \ref{hyp-base}(v), the functions $c_{kk}$ are bounded from above in $I\times \Rd$, hence a variant of the classical maximum principle (see \cite[Proposition 2.2]{AngLor10Com}) yields that $w_k$ is non positive in $I\times \Rd$. As a by product, taking $\f=-\bm{\one}$ in the definition of $w_k$, the claim follows.

{\em Step 2.} Here, we prove that for any $\delta\in (0,d_2-d_1)$ there exists a positive constant $K_{\delta}$ such that $(\G(t,s)(\varphi\bm{\one}))\le K_\delta\bm{\one}$ in $\Rd$ for any $(t,s)\in \Lambda_{[d_1,d_2]}$ with $t\ge s+\delta$.\\
Clearly, it suffices to prove the claim for $x$ outside a large enough ball. In view of this, we observe that since $h(x)\ge \tilde cx-\tilde a$ outside a suitable ball, for some positive constants $\tilde a$ and $\tilde c$, the arguments in Lemma \ref{pre_lem} can be applied to the function $\varphi$ and imply that $(\G(t,s)\varphi \bm{\one})(x)$ is well defined and
\begin{equation}\label{esse3}
(\G(t,s)\varphi(\bm{\one}))(x)- (\G(t,r)(\varphi\bm{\one}))(x)\ge -\int_s^t (\G(t,\sigma)(\A(\sigma)\varphi\bm{\one}))(x)d\sigma
\end{equation}
for any $r \le s\le t$ and $x\in \Rd$.
Now, let us fix $i\in\{1, \ldots, m\}$ and set $\mu_i(t,s,x,dy)=\sum_{j=1}^m p_{ij}(t,s,x,dy)$. Jensen inequality for
Borel finite measures and Step 1 yield that
\begin{align}\label{piove}
h_i\left((\G(t,s)(\varphi\bm{\one}))_i(x)\right)&=h_i\left(\int_{\Rd} \varphi(y)\mu_i(t,s,x,dy)\right)\notag\\
&\le \frac{1}{\mu_i(t,s,x,\Rd)}\int_{\Rd}h_i(\varphi(y))\mu_i(t,s,x,dy)\notag\\
&=\frac{1}{\mu_i(t,s,x,\Rd)}(\G(t,s)(h_i(\varphi)\bm{\one}))_i(x)\notag\\
& \le c_0^{-1}(\G(t,s)(h_i(\varphi)\bm{\one}))_i(x)
\end{align}
for any $d_1\le s\le t\le d_2$ and $x\in \Rd$, where in the last line we used equality $\mu_i(t,s,x,\Rd)=(\G(t,s)\bm{\one})_i(x)$ and estimate \eqref{bfb}.
Now, let us fix $x \in \Rd$, $t\in [d_1,d_2]$ and consider the functions $\beta_i:[0,t-\inf I)\to [0,+\infty)$ defined by
$\beta_i(\sigma)= (\G(t,t-\sigma)(\varphi{\bm \one}))_i(x)$,
for any $\sigma \in [0,t-\inf I)$. Then, from \eqref{esse3}, using also Hypothesis \ref{Lya3} and \eqref{piove}, we deduce that
\begin{align}
\beta_i (b)-\beta_i(0)&\le -\int_{t-b}^t (\G(t,\sigma)(h_i\circ\varphi))_i(x)d\sigma\notag\\
& \le -c_0 \int_{t-b}^t h_i((\G(t,\sigma)(\varphi))_i(x))d\sigma= -c_0\int_0^b h_i(\beta_i(\sigma))d\sigma,
\label{star-2}
\end{align}
where $b:=t-d_1$. From the previous chain of inequalities we can conclude that $\beta_i(r)\le y_i(r)$ for every $r\in [0,b]$, where
$y_i$ is the solution to the Cauchy problem
\begin{eqnarray*}
\left\{
\begin{array}{ll}
y'(r)=-c_0h(y(r)), & r\ge 0,\\[1mm]
y(0)=\varphi(x).
\end{array}
\right.
\end{eqnarray*}
Indeed, if this were not the case, we could determine $s_0\in (0,b)$ and an interval $J$ containing $s_0$ such that $\beta_i>y_i$ in $J$.
From \eqref{star-2}, written with the interval $(0,b)$ being replaced by $(s_1,s_2)$, we can infer that the function
$s\mapsto \beta(s)+c_0Ms$ is decreasing, where $M$ denotes the minimum of $h$ in $\R$. Therefore,
$\lim_{s\to s_0^-}(\beta(s)+c_0Ms)>\lim_{s\to s_0^-}(y(s)+c_0Ms)$ and this implies that $\beta$ is greater than $y$ in a left neighborhood of
$s_0$. Denoting by $\tau$ the infimum of $J$, then clearly, $\beta(\tau)=y(\tau)$. Writing \eqref{star-2} with $[0,b]$ being replaced by $[a,s]$, $s\in J$,
and observing that
\begin{eqnarray*}
y'(s)-y'(a)=-c_0\int_a^sh(y(r))dr
\end{eqnarray*}
we get
\begin{eqnarray*}
\beta(s)-y(s)\le c_0\int_a^s[h(y(r))-h(\beta(r))]dr,\qquad\;\,s\in J,
\end{eqnarray*}
which is clearly a contradiction since the left-hand side of the previous inequality is positive while its right-hand side is negative.

To conclude this step, it suffices to observe that $y$ is bounded from above in $[\delta,+\infty)$ for every $\delta>0$ as it can be easily checked writing
\begin{eqnarray*}
\int_{\varphi(x)}^{y(t)}\frac{dr}{h(r)}=-c_0t
\end{eqnarray*}
and using the integrability of $1/h$ in a neighborhood of $+\infty$.
Now, arguing as in the proof of \cite[Theorem 4.4]{AngLor10Com} we can prove that the functions $\beta_i$ are
bounded from above in $[\delta,b]$ for every $0<\delta<b$, uniformly with respect to $x\in \Rd$ and this proves the claim.

{\em Step 3.} Here, we show that the measures $\{p_{ij}(t,s,x,\cdot):\, x\in \Rd\}$ are tight for any
$(t,s)\in \Lambda_{[d_1,d_2]}$ and $i,j=1, \ldots, m$.
Let us fix $\varepsilon>0$. Then, arguing as in \eqref{loc_comp}, we can prove that there exists $R_0>0$ such that
\begin{align*}
0<p_{ij}(t,s,x,\Rd\setminus B_r)
=\left(\inf_{\Rd\setminus B_r}\varphi\right)^{-1}(\G(t,s)\varphi \bm{\one})_i(x)
\le K_\delta\left(\inf_{\Rd\setminus B_r}\varphi\right)^{-1}<K_\delta \varepsilon,
\end{align*}
for any $s,t\in \Lambda_{[d_1,d_2]}$ with $t\ge s+\delta$ and $r>R_0$, where we have taken into account that the family
$\{p_{ij}(t,s,x,\cdot): x\in \Rd, (t,s)\in \lambda_I\}$ are equivalent to the Lebesgue measure for any $i,j=1, \ldots, m$.
This implies that the family $\{p_{ij}(t,s,x,\cdot):\, x\in \Rd\}$ is tight for any $(t,s)\in \Lambda_{[d_1,d_2]}$,
with $t\ge s+\delta$ and $i,j=1, \ldots, m$.
The arbitrariness of $\delta$ allows to deduce the tightness of $p_{ij}(t,s,x,\cdot)$ for any $(t,s)\in \Lambda_{[d_1,d_2]}$
and $i,j=1, \ldots, m$ and, consequently, from Proposition \ref{equi-tight}, the compactness of $\G(t,s)$ in $C_b(\Rd;\Rm)$
for any $(t,s)\in \Lambda_{[d_1,d_2]}$. For the other values of $s,t$ the compactness of $\G(t,s)$ can be proved
by using the evolution law and the continuity of the operators $\G(t,s)$ in $\mathcal{L}(C_b(\Rd;\Rm))$. This completes the proof.

\end{proof}

\section{The action of the evolution operator $G(t,s)$ over some functional spaces}
Here, we study how the evolution operator $\G(t,s)$ acts over the spaces $C_0(\Rd;\Rm)$ of the continuous functions $\f:\Rd\to\Rm$ vanishing at infinity componentwise (i.e., $\lim_{|x|\to +\infty}f_i(x)=0$ for any $i=1, \ldots, m$), $L^p(\Rd;\Rm)$ and $C^1_b(\Rd;\Rm)$.

It is well known in the scalar case that the compactness property in the space of bounded and continuous functions is a sufficient condition which implies that
the spaces $C_0(\Rd)$ and $L^p(\Rd)$ are not preserved by action of the semigroup. Actually this is the case also for
the vector-valued evolution operator $\G(t,s)$ as we prove in the following.

\begin{thm}\label{thm_c0}
Under the assumptions of Theorem $\ref{thm_comp_2}$, the space $C_0(\Rd;\Rm)$
is not preserved by $\G(t,s)$ for any $(t,s)\in \Lambda_I$ with $s\le d_2$ and $t\ge d_1$.
On the other hand, if Hypotheses $\ref{hyp-base}(i)$-$(iv)$ and $\ref{allpos}$ hold true and there exist $\lambda_0>0$, $[a,b]\subset I$ and a function $\vv\in C^2(\Rd;\Rm)$, whose entries are all strictly
positive, vanishing at infinity and such that $\lambda_0 \vv-\A(t)\vv\ge 0$ for any $(t,x)\in [a,b]\times \Rd$, then
 $\G(t,s)(C_0(\Rd;\Rm))\subset C_0(\Rd;\Rm)$ for any $(t,s)\in \Lambda_{[a,b]}$.
\end{thm}
\begin{proof}
Let us fix $(t,s)\in \Lambda_I$ with $s\le d_2$ and $t\ge d_1$ ($i=1,\ldots, m$) and consider a sequence $(\f_n)\subset C_0(\Rd;\Rm)$ such that
$\chi_{B_n}{\bm{\one}}\le \f_n\le \chi_{B_{n+1}}{\bm{\one}}$ for any $n\in \N$.  Formula \eqref{int_rep_1}, estimate \eqref{stimaunif} and the compactness
of $\G(t,s)$ in $C_b(\Rd;\Rm)$ yield that $\G(t,s)\f_n$ converges uniformly in $\Rd$ to $\G(t,s)\bm{\one}$ as $n\to +\infty$.
Since $\G(t,s){\bm \one}$ is bounded from below by a positive constant (see Step 1 in the proof of Theorem \ref{thm_comp_2}),
it follows immediately that $\G(t,s)$ does not preserve $C_0(\Rd;\Rm)$.

Now, we prove the second part of the claim. Let $a,b$ and $\vv$ be as in statement and without loss of generality we can assume that $\lambda_0 \ge \max_{i=1,\ldots, m}\sum_{j=1}^m c_{ij}$ in order to apply Theorem \ref{maxprinc} to $\A(t)-\lambda_0\bm{I}$.
We begin by  proving that $\G(t,s)$ preserves the subset of $C_0(\Rd;\Rm)$ consisting of nonnegative
functions which belong to $C_c(\Rd;\Rm)$.
Let $\f\in C_c(\Rd;\Rm)$ be a nonnegative function and let $r>0$ be such that ${\rm supp}f_k\subset B_r$ for any $k=1,\ldots, m$.
The function $\zz(t,\cdot)= e^{-\lambda_0(t-s)}\uu(t,\cdot)-\delta^{-1} \|\f\|_\infty\vv$ where $\uu$ is the classical
solution of the problem \eqref{pb_Cauchy}, $\delta =\max_{k\in \{1,\ldots, m\}}\inf_{B_r}v_k$ being $\vv=(v_1, \ldots, v_m)$,
belongs to $C_b([s,T]\times \Rd)\cap C^{1,2}((s,T]\times \Rd)$ and solves the problem
\begin{equation*}
\left\{\begin{array}{ll}
D_t \zz(t,x)\le (\A(t)-\lambda_0\bm{I})\zz(t,x),\qquad\,\, &(t,x)\in(s,+\infty)\times \Rd,\\
\zz(s,x)\le 0,\qquad\;\, & x\in \Rd.
\end{array}\right.
\end{equation*}
Hence, Theorem \ref{maxprinc} can be applied to $\A(t)-\lambda_0\bm{I}$ to deduce that $\zz(t,x)\le 0$ in $[s,+\infty)\times \Rd$
or equivalently that $\bm{0}\le \uu \le e^{\lambda_0(t-s)}\delta^{-1}\|\f\|_\infty \vv$, which implies that $\uu$ belongs to $C_0(\Rd;\Rm)$.
Now, if $\f$ is not nonnegative then we can split $\f=\f^{+}-\f^{-}$ and,
arguing as above separately for $\f^+$ and $\f^-$, we deduce that the solutions $\uu^{\pm}$ of \eqref{pb_Cauchy} with $\f$ being replaced by $\f^{\pm}$ respectively,
belong to $C_0(\Rd;\Rm)$ as well as the solution $\uu=\uu^+-\uu^-$ of \eqref{pb_Cauchy}.
In the general case, we can argue by approximation. Indeed, let $\f$ be a bounded continuous function and $(\f_n)$ be a sequence
of $C_c(\Rd;\Rm)$ functions converging uniformly to $\f$ in $\Rd$. Then, since $\G(t,s)\f_n$ converges to
$\G(t,s)\f$ uniformly as $n\to +\infty$ for any $t\ge s$ we conclude also in this case.
\end{proof}

\begin{thm}\label{thm_lp}
The following statements hold true.
\begin{enumerate}[\rm (i)]
\item
Under the assumptions of Theorem $\ref{thm_comp_2}$, the space $L^p(\Rd;\Rm)$, $1\le p<+\infty$,
is not preserved by $\G(t,s)$ for any $(t,s)\in \Lambda_I$ with $s\le d_2$ and $t\ge d_1$.
\item
Let $q_{ij}^k\in C^{0,2}([a,b]\times\Rd)$ and $b_i^k\in C^{0,1}([a,b]\times\Rd)$, for any $i,j,l=1, \ldots, d$, $k=1,\ldots, m$ and some $[a,b]\subset I$, and
let $\kappa_C:[a,b]\times \Rd\to\R$ be any smooth function which bounds from above the quadratic form associated to the matrix $C$.
Further, suppose that
\begin{equation}\label{gammaab}
\Gamma_{[a,b]}:=\sup_{[a,b]\times \Rd}\Big (2\kappa_C-\min_{k=1,\ldots,m}{\rm div}_x{\bm\gamma}^k\Big )<+\infty,
\end{equation}
where ${\bm \gamma}^k:=(b_1^k-\sum_{j=1}^d D_jq_{1j}^k,\ldots,b_m^k-\sum_{j=1}^d D_jq_{mj}^k)$, $k=1, \ldots, m$.
Then, for any $p\ge 2$ and $(t,s)\in\Lambda_{[a,b]}$, $L^p(\Rd;\Rm)$ is invariant under $\G(t,s)$ and
\begin{equation}\label{l^p}
\|\G(t,s)\f\|_{L^p(\Rd;\Rm)}\le c_p(t-s)\|\f\|_{L^p(\Rd;\Rm)},
\end{equation}
where $c_p(r)= e^{[K(1-2/p)+\Gamma_{[a,b]}/p]r}$ and $K$ is defined in \eqref{moto}.
\item
Besides the assumptions in $(ii)$, assume that $q_{ij}^k\in C^{\alpha/2, 2+\alpha}_{\rm loc}([a,b]\times \Rd)$, $b_i^k\in C^{\alpha/2, 1+\alpha}_{\rm loc}([a,b]\times \Rd)$, for any $i,j=1\ldots,d$ and $k=1,\ldots, m$, and
\begin{equation}\label{p2}
\sup_{[a,b]\times \Rd}\bigg(\sum_{j=1}^mc_{jk}+\sum_{i,j=1}^d D_{ij}q_{ij}^k-\sum_{i=1}^dD_ib_i^k\bigg)<+\infty,\qquad\;\, k=1, \ldots, m.
\end{equation}
Then, estimate \eqref{l^p} can be extended to the case $p\in [1,2)$ taking
$c_p(r)=e^{[K^*(2/p-1)+\Gamma_{[a,b]}(1-1/p)]r}$ where $K^*\in\R$ is such that $\|\G^*(t,s)\|_{\mathcal{L}(C_b(\Rd;\Rm))}\le e^{K^*(t-s)}$ and $\G^*(t,s)$ is the adjoint operator of $\G(t,s)$.
\end{enumerate}
\end{thm}
\begin{proof}
(i) Let us fix $(t,s)\in \Lambda_I$ with $s\le d_2$ and $t\ge d_1$.
To prove that $L^p(\Rd;\Rm)$ ($p\in [1,+\infty)$) is not preserved by $\G(t,s)$, it suffices to consider the characteristic function $\chi_{B_R}$ where $R$ is
such that $\sum_{j=1}^m p_{ij}(t,s,x,\Rd\setminus B_R)\le c_0/2$, for any $i=1,\ldots, m$, and $c_0$ is defined in \eqref{bfb} (such a radius $R$ exists thanks
to the compactness of $\G(t,s)$ and Proposition \ref{tight-uni}).
Indeed, in this case, $\G(t,s)\chi_{B_R}\bm{\one}=\G(t,s)\bm{\one}-\G(t,s)(\chi_{\Rd\setminus B_R}\bm{\one})\ge c_0/2$ in $\Rd$ and consequently it does not belong to $L^p(\Rd;\Rm)$ for any $1\le p<\infty$.

(ii) To begin with, we notice that it suffices to prove the claim for nonnegative functions $\f$ belonging to $C_c(\Rd;\Rm)$. Indeed, for a general $\f\in C_c(\Rd;\Rm)$ we get the result simply writing $\f=\f^+-\f^-$ and observing that $|\f^{\pm}|\le|\f|$.
The case of an $L^p(\Rd;\Rm)$-function can be obtained by density.
Moreover, we observe that, if we prove \eqref{l^p} with $p=2$, then, thanks to the estimate \eqref{moto}, the Riesz-Thorin interpolation theorem yields estimate \eqref{l^p} for any $p\ge 2$ with $c_p(t-s)=[c_2(t-s)]^{2/p}e^{K(t-s)(1-2/p)}$ for any $(t,s)\in \Lambda_{[a,b]}$.
So, let us consider a nonnegative function $\f \in C_c(\Rd;\Rm)$ and prove that
\begin{equation}\label{lp-pre}
\|\mathbf{G}^{D}_R(t,s)\f\|_{L^2(B_R;\Rm)}\le e^{\Gamma_{[a,b]}(t-s)}\|\f\|_{L^2(\Rd;\Rm)},\qquad\;\, (t,s)\in \Lambda_{[a,b]},
\end{equation}
where  $\mathbf{G}^{D}_R(t,s)$ denotes the evolution operator associated to $\A(t)$ in $C(\overline{B_R};\Rm)$ with homogeneous Dirichlet boundary conditions.
Once \eqref{lp-pre} is proved, noticing that $\mathbf{G}^D_R(t,s)\f$ converges pointwise to $\G(t,s)\f$ as $R\to+\infty$, the Fatou lemma yields
\eqref{l^p} with $p=2$.

So, let us prove \eqref{lp-pre}. To simplify the notation we set $\uu_{R}(t,x):= (\G^{D}_{R}(t,s)\f)(x)$ for any $(t,s)\in \Lambda_{[a,b]}$ and $x\in\Rd$.
Using Hypothesis \ref{hyp-base} (ii) and the integration by parts formula we get
\begin{align*}
&\frac{d}{dt}\|\uu_{R}(t,\cdot)\|_{L^2(B_R;\Rm)}^2\\
=&2\sum_{k=1}^m\int_{B_R} u_{R,k}(t,\cdot)(\A(t)\uu_{R})_k(t,\cdot) dx\\
=&2\sum_{k=1}^m\sum_{i,j=1}^d\int_{B_R}q_{ij}^k(t,\cdot)u_{R,k}(t,\cdot)D_{ij}u_{R,k}(t,\cdot) dx\\
&+2\sum_{k=1}^m\sum_{i=1}^d\int_{B_R}b_{i}^k(t,\cdot)u_{R,k}(t,\cdot)D_{i}u_{R,k}(t,\cdot) dx
+2\int_{B_R}\langle C(t,\cdot)\uu_{R}(t,\cdot),\uu_{R}(t,\cdot)\rangle dx\\
\le&-\sum_{k=1}^m\sum_{i,j=1}^d\int_{B_R}(D_jq_{ij}^k(t,\cdot)-b_i^k(t,\cdot))D_{i}(u_{R,k}(t,\cdot))^2 dx+2\int_{B_R}\kappa_C(t,\cdot)|\uu_{R}(t,\cdot)|^2 dx\\
=&  -\sum_{k=1}^m\int_{B_R}{\rm div}_x \bm{\gamma}^k(t,\cdot) (u_{R,k}(t,\cdot))^2  dx+2\int_{B_R}\kappa_C(t,\cdot)|\uu_{R}(t,\cdot)|^2 dx\\
\le &\Gamma_{[a,b]}\int_{B_R}|\uu_{R}(t,\cdot)|^2 dx.
\end{align*}
Consequently, $\|\uu_{R}(t,\cdot)\|_{L^2(B_R;\Rm)}^2\le e^{\Gamma_{[a,b]}(t-s)}\|\f\|_{L^2(B_R;\Rm)}^2$, which gives the claim.

(iii) The additional assumptions in the statement allows us to apply Theorem \ref{bbthe} to the adjoint operator $\A^*(t)$.
This implies that the adjoint evolution operator $\{\G^*(t,s)\}_{t\ge s\in I}$ is well defined in $C_b(\Rd;\Rm)$ and satisfies the estimate
$\|\G^*(t,s)\|_{L(C_b(\Rd;\Rm))}\le e^{K^*(t-s)}$ for any $t\ge s\in I$ and some positive constant $K^*$. Moreover, the arguments in the proof of property (ii) show that
\begin{equation}
\|\G^*(t,s)\|_{L(L^q(\Rd;\Rm))}\le e^{[K^*(1-2/q)+\Gamma_{[a,b]}/q](t-s)},\qquad\;\, (t,s)\in\Lambda_{[a,b]},\;\,q\ge 2.
\label{polacchia}
\end{equation}

To complete the proof, it suffices to recall that
\begin{align*}
&\|\G(t,s)f\|_{L^p(\Rd;\Rm)}\\
=&\sup\bigg\{\int_{\Rd}\langle \f,\G^*(t,s)\g\rangle dx: \g\in C_c(\Rd;\Rm) \text{ and } \|\g\|_{L^{p'}(\Rd;\Rm)}\le 1\bigg\}
\end{align*}
for any $\f\in L^p(\Rd;\Rm)$ ($p\in [1,2)$) and use \eqref{polacchia}.
\end{proof}

Finally, we conclude this section investigating on the action of $\G(t,s)$ over the space $C^1_b(\Rd;\Rm)$. Theorem \ref{bbthe} states that the evolution operator maps the space $C_b(\Rd;\Rm)$ into $C_b(\Rd;\Rm)\cap C^1(\Rd;\Rm)$, but in general, $J_x\G(t,s)\f$ is not bounded whenever $\f$ belongs to $C^1_b(\Rd;\Rm)$. In the following Theorem \ref{stima1} we prove an uniform gradient estimate which answers to the question above.

\begin{hyp}\label{gra_est}
\begin{enumerate}[\rm (i)]
\item The coefficients $q_{ij}^k$, $b_i^k$ and $c_{ij}$ belong to $C^{\alpha/2, 1+\alpha}_{\rm loc}(I\times \Rd)$ for any $i,j=1,\ldots, d$ and $k=1, \ldots, m$;
\item there exist a positive constant $c$, $(m+2)$-functions $r_k:I\times\Rd\to \R$ $(k=1,\ldots,m)$ and $\rho_i:I\times\Rd\to (0,+\infty)$, $(i=0,1)$ such that
\begin{align}\label{ipo}
\;\;\;\;\;\;\;\;|\nabla_x q_{ij}^k|\leq c\mu_k, \qquad \langle J_x b^k\xi,\xi \rangle \le  r_k|\xi|^2,\qquad
|c_{hk}|\leq \omega_0\rho_0,\qquad
|\nabla_x c_{h'k'}|\le\omega_1 \rho_1
\end{align}
in $I\times \Rd$ for any $i,j=1, \ldots,d$, $h,h',k,k'=1, \ldots, m$, with $h\neq k$. In addition there exist two positive constants $\alpha_{k,J}$ and $\gamma_{k,J}$ such that
\begin{equation}
\label{defi}
\;\;\;\;\;\;\;\;\;\sigma_{k,J}:=\sup_{J\times\Rd}\left\{\left(\frac{d^2c^2}{4}-\alpha_{k,J}\right)\mu_k+r_k+c_{kk}+\gamma_{k,J}(\omega_0\rho_0^2+\omega_1\rho_1^2)\right\}<+\infty
\end{equation}
for any bounded interval $J\subset I$.
\end{enumerate}
\end{hyp}

\begin{thm}\label{stima1}
Assume that Hypotheses $\ref{gra_est}$ are satisfied. Then, for any $\f\in C^1_b(\Rd;\R^m)$ and $T>s$, the map $(s,T)\times \Rd\ni (t,x)\to |J_x (\G(t,s)\f)(x)|$ is bounded and satisfies the estimate
\begin{equation}\label{claim}
\||J_x \G(t,s)\f|\|_\infty\le \widetilde c_{s,T}\|\f\|_{C^1_b(\Rd;\Rm)},\qquad\;\,t\in (s,T),
\end{equation}
for some positive constant $\widetilde c$ depending on $s, T, m$, $\mu_k$ $($see Hypothesis $\ref{hyp-base}(ii))$ and $\sigma_{k,(s,T)}$ $(k=1,\ldots,m)$.
\end{thm}

\begin{proof}
Let $\f$ and $T$ be as in the statement and set $J=(s,T)$. We prove \eqref{claim} with $\G(t,s)$ being replaced by $\G^{N}_n(t,s)$, i.e., the evolution operator associated to $\A$ in $C_b(B_n;\Rm)$ with homogeneous Neumann boundary conditions. Then the claim will follow letting $n\to +\infty$ according to Remark \ref{Neu}.

For every $k=1,\ldots,m$, $t\in\overline J$, $x\in\overline B_n$, we set $v_{n,k}(t,x):=\alpha_{k,J}|\uu_n(t,x)|^2+|\nabla_xu_{n,k}(t,x)|^2$, where $u_{n,k}$ denotes the $k$-th component of $\G^N_n(\cdot,s)\f$.
A straightforward computation reveals that $\langle \nabla_x v_{n,k},\nu\rangle \le 0$ on $\partial B_n$. Indeed, taking into account the convexity of $B_n$ and the fact that $u_{n,k}$ satisfies homogeneous Neumann boundary conditions on $J\times\partial B_n$ we deduce that
\begin{align*}
\langle \nabla_x v_{n,k},\nu\rangle=&2\langle \nabla_x u_{n,k},\nu\rangle u_{n,k}+
2\langle D^2_x u_{n,k} \nabla_x u_{n,k}, \nu\rangle\\
= &2\left[\langle \nabla_x\langle \nabla_x u_{n,k}, \nu\rangle, \nabla_x u_{n,k}\rangle- \langle J\nu \nabla_x u_{n,k}, \nabla_x u_{n,k}\rangle\right]
\le  0
\end{align*}
on $J\times\partial B_n$. In addition, $v_{n,k}$ is a classical solution to the differential equation
$D_t v_{n,k}- \mathcal{A}_kv_{n,k}=2\sum_{i=1}^6\psi_i$
in $J\times B_n$ where $\mathcal{A}_k$ is defined in \eqref{op_A} and
\begin{align*}
\sum_{i=1}^6\psi_i:=&\sum_{i,j=1}^d\langle \nabla_x q_{ij}^k, \nabla_x u_{n,k}\rangle D_{ij}u_{n,k}+\langle J_x \bb^k \nabla_x u_{n,k},\nabla_x u_{n,k}\rangle\\
&+\sum_{j=1}^m\langle\nabla_x c_{kj},\nabla_x u_{n,k}\rangle u_{n,j}
+\sum_{j=1}^m c_{kj}\langle \nabla_x u_{n,k}, \nabla_xu_{n,j}\rangle\\
&-\alpha_{k,J}\langle Q^k\nabla_xu_{n,k},\nabla_xu_{n,k}\rangle-\sum_{i,j=1}^d q_{ij}^k\langle \nabla_x D_i u_{n,k}, \nabla_x D_j u_{n,k}\rangle
\end{align*}
in $(s,T)\times\Rd$. Using the Cauchy-Schwartz inequality, estimates \eqref{ipo} and Hypothesis \ref{hyp-base}(ii) we can estimate the terms in $\psi_i$ $(i=1,\ldots, 5)$ as follows:
\begin{align*}
&\psi_1\le dc\mu_k|\nabla_x u_{n,k}||D^2_x u_{n,k}| \le dc\mu_k \Big(\varepsilon|D^2_x u_{n,k}|^2+\frac{1}{4\varepsilon}|\nabla_x u_{n,k}|^2 \Big),\\[1mm]
&\psi_2\le  r_k|\nabla_x u_{n,k}|^2,\\[1mm]
&\psi_3\le \sqrt{m}\omega_1 \rho_1|\nabla_x u_{n,k}||\uu_n|\le \varepsilon_1\rho_1^2|\nabla_x u_{n,k}|^2+ \frac{m}{4\varepsilon_1}\omega_1^2|\uu_n|^2\\[1mm]
&\psi_4\le c_{kk}|\nabla_x u_{n,k}|^2+ \omega_0\rho_0|\nabla_x u_{n,k}||J_x\uu_n|
\le (c_{kk}+\varepsilon_2\rho_0^2)|\nabla_x u_{n,k}|^2+\frac{m}{4\varepsilon_2}\omega_0^2|J_x\uu_n|,\\[1mm]
&\psi_5+\psi_6\le -\mu_k(\alpha_{k,J}|\nabla_x u_{n,k}|^2+|D^2_xu_{n,k}|^2).
\end{align*}
Hence, we deduce that
\begin{align*}
\sum_{i=1}^6\psi_i \le &\mu_k(dC \varepsilon -1)|D^2_x u_{n,k}|^2+\left(\frac{dc}{4\varepsilon}\mu_k\!+\!r_k\!+\!\varepsilon_1\rho_1^2\!+\!\varepsilon_2\rho_0^2\!+\!c_{kk}\!-\!\alpha_{k,J}\mu_k\right)|\nabla_x u_{n,k}|^2\\
&+\frac{m}{4\varepsilon_1}\omega_1^2|\uu_n|^2+ \frac{m}{4\varepsilon_2}\omega_0^2|J_x\uu_n|^2.
\end{align*}

Choosing $\varepsilon=(dC)^{-1}$, $\varepsilon_1=\varepsilon_2=\gamma_k$ and using \eqref{defi} we conclude that
\begin{eqnarray*}
\sum_{i=1}^6\psi_i\le \sigma_{k,J} v_{n,k}+\frac{m}{4\gamma_k}(\omega_0^2\vee \omega_1^2)(|\uu_n|^2+|J_x\uu_n|^2).
\end{eqnarray*}
A variant of the classical maximum principle shows that
\begin{align*}
v_{n,k}(t,\cdot)\le &\tilde{G}^{N}_{n,k}(t,s)(\alpha_{k,J}|f_k|^2+|\nabla f_k|^2)\\
&+\frac{m}{4\gamma_k}(\omega_0^2\vee \omega_1^2)\int_s^t \tilde{G}^{N}_{n,k}(t,r)(|\uu_n(r,\cdot)|^2+|J_x\uu_n(r,\cdot)|^2)dr
\end{align*}
for any $t\in J$, where $\tilde{G}^{N}_{n,k}(t,s)$ denotes the evolution operator associated to the operator ${\mathcal A}_k+\sigma_{k,J}$ in $C(\overline B_n)$ with homogeneous Neumann boundary conditions. Taking into account that $\|\tilde{G}^{N}_{n,k}(t,s)\|_{\mathcal L(C(\overline{B}_n))}\le e^{\sigma_{k,J}(t-s)}$ for any $t>s\in I$, we can estimate
\begin{align*}
\|\nabla_x u_{n,k}(t,\cdot)\|_\infty^2\le&  e^{\sigma_{k,J}(t-s)}(\alpha_{k,J}|f_k|^2+\|\nabla f_k\|_\infty^2)\\
&+\frac{m}{4\gamma_k}(\omega_0^2\vee \omega_1^2)\int_s^t e^{\sigma_{k,J}(t-r)}(\|\uu_n(r,\cdot)\|_{\infty}^2+\|J_x\uu_n(r,\cdot)\|_\infty^2)dr\\
 \le& e^{c_{0,J}^+(t\!-\!s)}\bigg (\|\nabla f_k\|_\infty^2\!+\!c_{1,J}(t-s)\|\f\|_\infty^2\!+\!c_{1,J}\int_s^t \|J_x\uu_n(r,\cdot)\|_\infty^2dr\bigg ),
\end{align*}
for any $t\in J$, where $c_{0,J}=\displaystyle\max_{k=1,\ldots,m}\sigma_{k,J}$ and $c_1=\displaystyle\Big (4\min_{k=1,\ldots,m}\gamma_{k,J}\Big )^{-1}(\omega_1^2\vee \omega_2^2)$. Summing over $k$ from $1$ to $m$ we deduce that
\begin{eqnarray*}
\|J_x\uu_n(t,\cdot)\|^2_\infty\le \overline c\bigg (\|J \f\|_\infty^2+\|\f\|_\infty^2 +\int_s^t \|J_x\uu_n(r,\cdot)\|_\infty^2dr\bigg ),\qquad\;\,t\in J,
\end{eqnarray*}
 and $\overline c$ is a positive constant depending on $s$, $T$, $m$, $\omega_0$, $\omega_1$, $\gamma_k$ ($k=1,\ldots,d$) and $c_0$. Applying Gronwall lemma, we conclude the proof.
\end{proof}

\section{Invariant measures}\label{sect-inv}
In this section we prove the existence of evolution systems of invariant measures associated to $\G(t,s)$, i.e., a family of positive and finite Borel measures over $\Rd$, $\{\mu_{i,r}:\,r\in I, i=1,\ldots, m\}$ such that
\begin{equation}\label{esm}
\sum_{i=1}^m \int_{\Rd} (\G(t,s)\f)_id{\bf \mu}_{i,t}= \sum_{i=1}^m \int_{\Rd}f_id{\bf \mu}_{i,s}
\end{equation}
for any $\f\in C_b(\Rd;\Rm)$ and any $I\ni s<t$. To this aim, the results in Section \ref{comp-sect} and in particular Theorem \ref{thm_comp1} are crucial.
Here we assume that Hypotheses \ref{hyp-base}(i)-(iii) and \ref{allpos} are satisfied.

\begin{prop}
\label{prop-rustico}
Let $\{\mu_{i,r}:\,r\in I, i=1,\ldots, m\}$ be a family of nonnegative and finite Borel measures which satisfy condition \eqref{esm}. Then, all the measures of the family are either trivial or equivalent to the Lebesgue measure. As a byproduct, formula \eqref{esm} can be extended to the set of all the bounded Borel measurable functions.
\end{prop}

\begin{proof}
We assume that the measures of the family are not all the trivial measure. Thus, we can fix $i\in\{1,\ldots,m\}$ and $r\in I$ such that $\mu_{i,r}(\Rd)>0$. To improve the readability, we split the proof into two steps.

{\em Step 1.} Here, we prove that the measures of the family are all positive. We begin by fixing $j\in\{1,\ldots, m\}$, $s\in I$ smaller than $r$.
Writing formula \eqref{esm} with $\f={\bf e}_j$ gives
\begin{equation}
\mu_{j,s}(\Rd)=\int_{\Rd}d\mu_{j,s}=\sum_{k=1}^m\int_{\Rd}(\G(r,s){\bf e}_j)_kd\mu_{k,r}\ge\int_{\Rd}(\G(r,s){\bf e}_j)_id\mu_{i,r}.
\label{star-1}
\end{equation}
Since the function $\G(r,s){\bf e}_j$ is strictly positive in $\Rd$, thanks to Proposition \ref{pro}, and $\mu_{i,r}$ is a positive measure,
it follows immediately that the last side of \eqref{star-1} is positive as well. Hence, $\mu_{j,s}(\Rd)$ is positive as it has been claimed.

Next, we fix $s_1<r$ and use again formula \eqref{esm} to write
\begin{eqnarray*}
\sum_{k=1}^m\int_{\Rd}(G(2r,s_1){\bf e}_j)_kd\mu_{k,2r}=\mu_{j,s_1}(\Rd)>0.
\end{eqnarray*}
Since $(G(2r,s_1){\bf e}_j)_k>0$ in $\Rd$ for any $k\in\{1,\ldots,d\}$, there should exist an index $k_0$ such that
$\mu_{k_0,2r}(\Rd)>0$. Hence, the same argument used above with $(k_0,2r)$ replacing $(i,r)$ shows that $\mu_{j,s}$ is a positive measure for any $s<2r$. Iterating this argument, we can prove that all the measures of the family are positive.

{\em Step 2.} To prove that the measures $\mu_{j,t}$ ($j=1,\ldots,m$, $t\in I$) are equivalent to the Lebesgue measure, we need to extend the validity of
\eqref{esm} to the case when $\f=\chi_A{\bf e}_j$ ($j=1,\ldots,m$) and $A$ is a Borel subset of $\Rd$.
For this purpose, we begin by assuming that $A$ is an open set and denote by $(\theta_n)$ a sequence of
continuous functions converging to $\chi_A$ pointwise in $\Rd$ and such that $0\le\theta_n\le 1$ for any $n\in\N$ (see Lemma \ref{lemma-app}). By the last part of Theorem \ref{teo}, we know that $\G(t,s)(\vartheta_n{\bf e}_j)$ converges to $\G(t,s){\bf e}_j$ as $n\to +\infty$, for any $I\ni s<t$, and $\|\G(t,s)(\vartheta_n{\bf e}_j)\|_{\infty}\le 1$. Therefore, writing \eqref{esm} with $\f=\vartheta_n{\bf e}_j$ and letting $n$ tend to $+\infty$, we conclude that
\begin{eqnarray*}
\sum_{k=1}^m\int_{\Rd}(\G(t,s)(\chi_A{\bf e}_j))_kd\mu_{k,t}=\mu_{j,s}(A).
\end{eqnarray*}
We now observe that the function $\nu_t$, defined by
\begin{equation*}
\nu_t(A)=\sum_{k=1}^m\int_{\Rd}(\G(t,s)(\chi_A{\bf e}_j))_kd\mu_{k,t}
\end{equation*}
for any Borel set $A$, is a nonnegative measure since $\G(t,s)(\chi_A{\bf e}_j)\ge 0$ for any Borel set $A$.
Moreover, it agrees with $\mu_{j,s}$ on the open sets of $\Rd$, which generate the $\sigma$-algebra of all the Borel subsets of $\Rd$. Hence,
$\mu_{j,s}$ and $\nu_t$ are actually the same measure and it follows that
\begin{eqnarray*}
\sum_{k=1}^m\int_{\Rd}(\G(t,s)(\chi_A{\bf e}_j))_kd\mu_{k,t}=\mu_{j,s}(A),\qquad\;\,I\ni s<t,\;\,j=1,\ldots,m,
\end{eqnarray*}
for any Borel set $A$, as it has been claimed.
From this formula the equivalence of the Lebsegue measure and each measure $\mu_{j,s}$ follows. Indeed,
since the measures $\mu_{k,t}$ and $\mu_{j,s}$ are positive and the function $\G(t,s)(\chi_A{\bf e}_j)$ is nonnegative, it easy to infer that
$\mu_{j,s}(A)=0$ if and only if $(\G(t,s)(\chi_A{\bf e}_j))_k=0$ in $\Rd$ for any $k=1,\ldots,m$. But, since each measure $p_{kh}(s+1,s,x,dy)$ is positive and equivalent to the Lebesgue measure (see again Proposition \ref{pro}), this is the case if and only if $A$ has zero Lebesgue measure.

To complete the proof, it suffices to observe that for any bounded Borel measurable function $\f$ there exists a sequence $(\f_n)$ of bounded and continuous functions converging to $\f$ almost everywhere (with respect to the Lebesgue measure and, hence, with respect to each measure $\mu_{j,t}$ of the family) as $n$ tends to $+\infty$.
Clearly,
\begin{eqnarray*}
\lim_{n\to +\infty}\sum_{k=1}^m\int_{\Rd}f_{n,k}d\mu_{k,s}
=\sum_{k=1}^m\int_{\Rd}f_kd\mu_{k,s}
\end{eqnarray*}
and the sequence $(\G(t,s)\f_n)$ is bounded in $C_b(\Rd)$ and converges to $\G(t,s)\f$ pointwise in $\Rd$.
Thus, writing \eqref{esm} with $\f$ being replaced by $\f_n$ and letting $n$ tend to $+\infty$, we extend the validity of such a formula to $\f\in B_b(\Rd;\Rm)$.
\end{proof}

\begin{lemm}\label{pos}
The following properties hold true:
\begin{enumerate}[\rm (i)]
\item
Under Hypothesis $\ref{hyp-base}$ and $\ref{Lya2}$, if there exist $j\in \{1,\ldots, m\}$ and a positive function $g \in C_b(\Rd)\cap C^2(\Rd)$ such that $(\mathcal{A}_j(t)g)(x)+c_{jj}(t,x)g(x) \ge 0$ for any $t \in I$ and $x \in \Rd$, then $(\G(\cdot,s)g\e_j)_j\ge g_j$ in $(s,+\infty)\times\Rd$;
\item
Under Hypotheses $\ref{hyp-base}$, assume further that $\sum_{j=1}^mc_{ij}\le 0$ on $\Rd$ for every
$i=1,\ldots,m$ and that there exist a positive function $\g\in C_b(\Rd;\Rm)\cap C^2(\Rd;\Rm)$ such that $\A(t)\g\ge 0$  in $\Rd$ for any $t\in I$. Then, $\G(t,s)\g\ge \g$ in $\Rd$ for any $t>s\in I$.
    \end{enumerate}
\end{lemm}
\begin{proof}
(i) A direct computation reveals that the function $v_j:= (\G(\cdot,s)g\e_j)_j- g$ belongs to $C^{1,2}((s,+\infty)\times \Rd)\cap C([s,+\infty)\times \Rd)$ and
solves the problem
\begin{equation*}
\left\{\begin{array}{ll}
D_t v_j(t,x)\ge (\mathcal{A}_j(t)v_j)(x) + c_{jj}(t,x)v_j(t,x),\qquad\;\, &(t,x)\in (s,+\infty)\times \Rd,\\
v_j(s,x)=0,\qquad\;\, &x\in \Rd.
\end{array}
\right.
\end{equation*}
Observing that Hypothesis \ref{Lya2} yields the existence of a Lyapunov function for the operator $\mathcal A_j$ (hence for $\mathcal A_j+c_{jj}$) and invoking a generalization of the classical maximum principle (see \cite[Proposition 2.2]{AngLor10Com}) we deduce that $v_j \ge 0$ in $(s,+\infty)\times \Rd$ and we are done.\\
(ii) The claim can be obtained immediately just applying the maximum principle in Proposition \ref{maxprinc} to the function $\vv=\G(\cdot,s)\g-\g$.
\end{proof}

\begin{thm}\label{thm_esm}
Under Hypotheses $\ref{Lya2}$, $\ref{Lya3}$, if $c_{ij}\ge 0$ for every $i,j\in\{1,\ldots,m\}$, with $i\neq j$ and the hypotheses of Lemma $\ref{pos}(i)$ or $(ii)$ hold true, then there exists an evolution system of measures associated with the evolution operator $\G(t,s)$. Each measure of this system is positive and equivalent to the Lebesgue one.
\end{thm}

\begin{proof}
We fix $j\in\{1,\ldots,m\}$, $x_0\in\Rd$, $n\in\N$ and, for any $r\in\N$ with $r>n$, we consider the family of measures $\{p^{r,n}_i: r>n,\, i=1,\ldots,m\}$ defined by
\begin{eqnarray*}
p^{r,j}_{i,n}(A)=\frac{1}{r-n}\int_n^rp_{ji}(\tau,n,x_0,A)d\tau,\qquad\;\, A\in {\mathcal B}(\Rd).
\end{eqnarray*}

By Corollary \ref{tight-uni}, each family $\{p^{r,j}_{i,n}: r>n\}$ is tight. Therefore, we can invoke a generalization of Prokhorov's theorem (see e.g., \cite[Theorem 8.6.2]{bogachev}) to infer that, up to a subsequence, $\{p^{r,j}_{i,n}: r>n\}$ weakly$^*$ converges to some measure $\mu^j_{i,n}$ as $r \to +\infty$, i.e.,
\begin{eqnarray*}
\lim_{r\to+\infty}\frac{1}{r-n}\int_n^r(\G(\tau,n)(f{\bf e}_i))_j(x_0)d\tau=\lim_{r\to+\infty}\int_{\Rd}fdp^{r,j}_{i,n}=\int_{\Rd}fd\mu^j_{i,n}
\end{eqnarray*}
for any $f\in C_b(\Rd)$.

By a diagonal argument, we can extract an increasing sequence $(r_k)$ of integers such that
$p^{r_k,j}_{i,n}$ weakly$^*$ converges to $\mu^j_{i,n}$ as $k$ tends to $+\infty$, for each $n\in\N$. As a byproduct, we can infer that
\begin{align}\label{app_g}
\lim_{k\to+\infty}\frac{1}{r_k-n}\int_n^{r_k}(\G(\tau,n)\f)_j(x_0)d\tau=\sum_{i=1}^m\int_{\Rd}f_id\mu^j_{i,n},\qquad\;\,\f\in C_b(\Rd;\Rm).
\end{align}
Writing formula \eqref{app_g} with $\f$ being replaced by $g \e_l$ (resp. $\g$), if the assumptions of Lemma \ref{pos}(i) (resp. (ii)) are satisfied, yields immediately that $\mu^j_{l,n}$ is not the trivial measure. Indeed in the first case $\liminf_{k\to +\infty}(r_k-n)^{-1}\int_n^{r_k}(\G(\tau,n)(g \e_l))_j(x_0)d\tau>0$  and in the second one $\liminf_{k\to +\infty}(r_k-n)^{-1}\int_n^{r_k}(\G(\tau,n)\g)_j(x_0)d\tau>0$ for any $l=1, \ldots, m$.
Moreover, for any $\f\in C_b(\Rd;\Rm)$ and $h,n\in\N$ with $h>n$ we can write
\begin{align}
\sum_{i=1}^m\int_{\Rd}(\G(h,n)\f)_id\mu^j_{i,h}=&\lim_{k\to+\infty}\frac{1}{r_k-h}\int_h^{r_k}(\G(\tau,h)\G(h,n)\f)_j(x_0)d\tau\notag\\
=&\lim_{k\to+\infty}\frac{1}{r_k-h}\int_h^{r_k}(\G(\tau,n)\f)_j(x_0)d\tau\notag\\
=&\lim_{k\to+\infty}\frac{1}{r_k-n}\int_n^{r_k}(\G(\tau,n)\f)_j(x_0)d\tau\notag\\
-&\lim_{k\to +\infty}\frac{1}{r_k-n}\int_n^{h}(\G(\tau,n)\f)_j(x_0)d\tau\notag\\
+&\lim_{k\to +\infty}\frac{h-n}{(r_k-h)(r_k-n)}\int_h^{r_k}(\G(\tau,n)\f)_j(x_0)d\tau\notag\\
=&\sum_{i=1}^m\int_{\Rd}f_id\mu^j_{i,n}.
\label{uscira}
\end{align}

Now, we define the measures $\mu^j_{i,s}$ also for non integer values of $s$. For this, purpose, we set
\begin{align*}
\mu^j_{i,s}(A)=\sum_{k=1}^m\int_{\Rd}(\G(n,s)(\chi_A{\bf e}_i))_kd\mu^j_{k,n},\qquad\;\,A\in {\mathcal B}(\Rd),
\end{align*}
where $n$ is any integer larger than $s$. It is straightforward to check that $\mu^j_{i,s}$ is a nonnegative measure
and that
\begin{align*}
\int_{\Rd}fd\mu^j_{i,s}=\sum_{k=1}^m\int_{\Rd}(\G(n,s)(f{\bf e}_i))_kd\mu^j_{k,n}
\end{align*}
for any $f\in C_b(\Rd)$, so that
\begin{align*}
\sum_{i=1}^m\int_{\Rd}f_id\mu^j_{i,s}=\sum_{k=1}^m\int_{\Rd}(\G(n,s)\f)_kd\mu^j_{k,n},\qquad\;\,\f\in C_b(\Rd;\Rm).
\end{align*}
Note that the above definition is independent of the choice of $n>s$. Indeed, if $p$ is another integer larger than $s$ (to fix the ideas we suppose that $p>n$) then splitting $\G(p,s)(\chi_A{\bf e}_i)=\G(p,n)\G(n,s)(\chi_A{\bf e}_i)$ and using \eqref{uscira}, we conclude that
\begin{align*}
\sum_{k=1}^m\int_{\Rd}(\G(p,s)(\chi_A{\bf e}_i))_kd\mu^j_{k,p}
=\sum_{k=1}^m\int_{\Rd}(\G(n,s)(\chi_A{\bf e}_i))_kd\mu^j_{k,n},
\end{align*}
which shows that the measure $\mu^j_{i,s}$ is well defined.

To prove the invariance of the system $\{\mu^j_{i,s}: s\in I, i=1,\ldots,m\}$, we fix $t>s\in I$, $n>t$ and observe that
\begin{align*}
\sum_{k=1}^m\int_{\Rd}f_kd\mu^j_{k,s}=&\sum_{k=1}^m\int_{\Rd}(\G(n,s)\f)_kd\mu^j_{k,n}
=\sum_{k=1}^m\int_{\Rd}(\G(n,t)\G(t,s)\f)_kd\mu^j_{k,n}\\
=&\sum_{k=1}^m\int_{\Rd}(\G(t,s)\f)_kd\mu^j_{k,t}
\end{align*}
for any $\f\in C_b(\Rd;\Rm)$.

The equivalence of each measure $\mu^j_{i,s}$ with respect to the Lebesgue measure and its positivity are immediate consequence of Proposition \ref{prop-rustico}. Indeed it suffices to observe that the evolution system of measures $\{\mu^j_{i,s}: i=1,\ldots,m,\, s\in I\}$ contains at least a non trivial measure.
\end{proof}

\subsection{The evolution operator $\bm{\G(t,s)}$ in $\bm{L^p}$-spaces}
In this subsection, we prove that the evolution operator $\G(t,s)$ can be extended, with a bounded semigroup in the $L^p$-spaces related to evolution system of measures and, in the autonomous case, assuming compactness in $C_b(\Rd;\Rm)$ we prove compactness in these $L^p$-spaces too.

Here, we consider $\{\mu_{i,t}: t\in I,\, i=1,\ldots,m\}$ which is any evolution system of measures associated to $\G(t,s)$.
Moreover, for any $p\in [1,+\infty)$, we write $L^p_{{\bm\mu}_t}(\Rd;\Rm)$ to denote the set $\bigotimes_{i=1}^mL^p_{\mu_{i,t}}(\Rd)$, which we endow with
the natural norm $\f\mapsto\left (\sum_{i=1}^m\int_{\Rd}|f_i|^pd\mu_{i,t}\right )^{1/p}=:\|\f\|_{L^p_{{\bm\mu}_t}}$.
For $p=\infty$, the space $L^\infty_{{\bm\mu}_t}(\Rd;\Rm)$ denotes the set of all $\mu_t$-essentially bounded functions $\f$ with norm $\|f\|_{L^\infty_{{\bm\mu}_t}(\Rd;\Rm)}=\max_{k=1,\ldots, m}{\rm ess sup}_{x\in \Rd}|f_k(x)|$. Note that, in view of Proposition \ref{prop-rustico}, the measures $\mu_{i,t}$ ($t\in I$ and $i=1,\ldots,m$) are all equivalent to the Lebesgue measure. Thus, the Lebesgue space $L^\infty(\Rd;\Rm)$ equals to $L^\infty_{{\bm\mu}_t}(\Rd;\Rm)$ for any $t\in I$.


\begin{prop}
Each $\G(t,s)$ can be extended with a bounded operator mapping $L^p_{{\bm\mu}_s}(\Rd;\Rm)$ into $L^p_{{\bm\mu}_t}(\Rd;\Rm)$ for any $1\le p<+\infty$ which satisfies the estimate
\begin{align}
\label{festa}
\|\G(t,s)\|_{\mathcal L(L^p_{{\bm\mu}_s}(\Rd;\Rm),L^p_{{\bm\mu}_t}(\Rd;\Rm))}\leq (2e^{K(t-s)})^{\frac{p-1}{p}}, \qquad\;\, t>s,
\end{align}
for any $p \in [1, +\infty)$, where $K$ is defined in \eqref{moto}.
\end{prop}

\begin{proof}
Since $\|\G(t,s)\e_k\|_{\infty}\leq e^{K(t-s)}$, it follows that $p_{ik}(t,s,x,\Rd)\leq e^{K(t-s)}$ for any $i,k=1,\ldots,m$, $t>s\in I$ and $x\in\Rd$.
Thus, the Jensen inequality and formula \eqref{int_rep_1} yield
\begin{align*}
|(\G(t,s)\f)_i(x)|^p
\leq & 2^{p-1}\sum_{k=1}^m\left|\int_{\Rd}f_k(y)p_{ik}(t,s,x,dy)\right|^p\\
\leq & 2^{p-1}\sum_{k=1}^m [p_{ik}(t,s, x, \Rd)]^{p-1}\int_{\Rd}\left|f_k(y)\right|^pp_{ik}(t,s,x,dy) \\
\leq & 2^{p-1}e^{K(p-1)(t-s)}(\G(t,s)(|f_1|^p, \ldots,|f_m|^p ))_i(x)
\end{align*}
for any $t>s, x\in \Rd$, $i=1, \ldots, m$, $\f\in C_b(\Rd;\R^m)$ and $p \in[1,+\infty)$. Moreover, from the invariance property \eqref{esm}, we deduce that
\begin{align*}
\sum_{i=1}^m\int_{\Rd}|(\G(t,s)\f)_i|^pd\mu_{i,t}
\leq & 2^{p-1}e^{K(p-1)(t-s)}\sum_{i=1}^m\int_{\Rd}(\G(t,s)(|f_1|^p, \ldots,|f_m|^p ))_id\mu_{i,t} \\
=& 2^{p-1}e^{K(p-1)(t-s)}\sum_{i=1}^m\int_{\Rd}|f_i|^pd\mu_{i,s}
\end{align*}
for any $t>s$ and $\f\in C_b(\Rd;\Rm)$. Since the measures $\mu_{i,t}$ ($i=1, \ldots,m$, $t\in I$) are finite Borel measures, the space $C_b(\Rd;\Rm)$ is dense in $L^p_{{\bm\mu}_t}(\Rd;\Rm)$ for any $p \in[1,+\infty)$ and $t\in I$.
(see \cite[Remark 1.46]{AFP}), hence, from the previous chain of inequalities we easily deduce that $\G(t,s)$ extends
to a linear bounded operator from $L^p_{\bm\mu_s}(\Rd;\Rm)$ into $L^p_{\bm\mu_t}(\Rd;\Rm)$ and formula \eqref{festa} follows. The evolution property easily follows. Hence, $\G(t,s)$ is an evolution operator from $L^p_{\bm\mu_s}(\Rd;\Rm)$ into $L^p_{\bm\mu_t}(\Rd;\Rm)$.
\end{proof}

\begin{rmk}{\rm
In the autonomous case, the evolution operator $\G(t,s)$ is replaced by a semigroup $\T(t)$ and the evolution system of measures $\{\mu_{i,t}: i=1,\ldots, m,\,t\in I\}$ is replaced by a family of measures not depending on the parameter $t$ denoted by $\{\mu_i:\,i=1, \ldots, m\}$. In this case the semigroup $\T(t)$ maps $L^p_{{\bm{\mu}}}(\Rd;\Rm)$ into itself and $\|\T(t)\|_{\mathcal L(L^p_{{\bm\mu}}(\Rd;\Rm))}\leq (2e^{Kt})^{\frac{p-1}{p}}$, for any $t>0$ and $p\in [1,+\infty)$. In addition, $\T(t)$ turns out to be a strongly continuous semigroup in $L^p_{{\bm{\mu}}}(\Rd;\Rm)$. Indeed, for any $\f \in C_b(\Rd;\Rm)$, $\T(t)\f$ converges locally uniformly to $\f$ as $t\to 0^+$. Hence, estimate \eqref{moto}, Proposition \ref{prop-rustico} and the dominated convergence theorem allow us to conclude that $\|\T(t)\f-\f\|_{L^p_{{\bm{\mu}}}(\Rd;\Rm)}$ vanishes as $t\to 0^+$. For $\f\in L^p_{{\bm{\mu}}}(\Rd;\Rm)$ we can get the same result using the density of $C_b(\Rd;\Rm)$ in $L^p_{{\bm{\mu}}}(\Rd;\Rm)$ and the boundedness of the function $t\mapsto\|\T(t)\|_{\mathcal L(L^p_{{\bm\mu}}(\Rd;\Rm))}$ in $(0,1)$.}
\end{rmk}

Now, we give a sufficient condition in order that the evolution operator $\G(t,s)$ is compact from $L^p_{\bm\mu_s}(\Rd;\Rm)$ into $L^p_{\bm\mu_t}(\Rd;\Rm)$.

\begin{thm}
Assume that $\G(t_0,s)$ is compact in $C_b(\Rd;\Rm)$ for some $I \ni s<t_0$. Then, $\G(t_0,s))$ is compact from $L^p_{\bm\mu_s}(\Rd;\Rm)$ into $L^p_{\bm\mu_t}(\Rd;\Rm)$ for any $p>1$.
\end{thm}
\begin{proof}
Let us fix  $t_0>s\in I$ and assume that $\G(t_0,s)$ is compact in $C_b(\Rd;\Rm)$. First of all, we show that $\G(t_0,s)$ is compact in $L^\infty(\Rd;\Rm)= L^\infty_{\bm\mu_s}(\Rd;\Rm)$ for any $s\in I$, where the equality follows from Proposition \ref{prop-rustico}. Since the evolution operator is strong Feller, $\G(t_0,s)$ maps $L^\infty(\Rd;\Rm)$ into $C_b(\Rd;\Rm)$. Moreover, by the semigroup law and the compactness in $C_b(\Rd;\Rm)$, $\G(t_0,s)$ turns out to be compact from $L^\infty(\Rd;\Rm)$ into $C_b(\Rd;\Rm)$ hence from $L^\infty(\Rd;\Rm)$ into itself.

Now, let $U$ be the unit ball in $L^\infty(\Rd;\Rm)$, set $K^{t_0,s}:=\G(t_0,s)(U)$ and fix $\varepsilon>0$. Thanks to the compactness of $\G(t_0,s)$ we can determine simple vector-valued functions $\{\bm \zeta_j\}_{j=1,\ldots, k}$, with $\bm\zeta_j=\sum_{i=1}^{n}\bm{c}^j_i\chi_{A_i}$ for some $\bm{c}^j_i\in \R^m$, $n\in\N$, where $\cup_{i=1}^{n}A_i=\Rd$, such that the family $\{\bm \zeta_1,\ldots, \bm \zeta_k\}$ is an $\varepsilon$-net for $K^{t_0,s}$, i.e., $K^{t_0,s}\subset \bigcup_{i=1}^k B_\varepsilon(\bm \zeta_i)$. Moreover, $P_\varepsilon^t \bm\zeta_i=\bm\zeta_i$ for any $i=1, \ldots, k$ and $t\in I$, where
\begin{equation*}
(P_\varepsilon^t \f)_{\ell}= \sum_{i=1}^n\left( \frac{1}{\mu_{{\ell},t}(A_i)}\int_{A_i}f_{\ell} d\mu_{\ell,t}\right)\chi_{A_i},\qquad\;\, \ell=1, \ldots, m,\,\,t\in I.
\end{equation*}
Note that
\begin{equation}\label{1}
\|P_\varepsilon^{t_0} \G(t_0,s)-\G(t_0,s)\|_{\mathcal{L}(L^\infty(\Rd;\Rm))}\le 2\varepsilon.
\end{equation}
Indeed, fix $\f\in L^\infty(\Rd;\Rm)$ with $\|\f\|_\infty\le 1$. Then, there exists $j\in \{1, \ldots, k\}$ such that $\G(t_0,s)\f\in B_\varepsilon(\bm \zeta_j)$. Hence,
\begin{equation*}
\|P_\varepsilon^{t_0} \G(t_0,s)\f-\G(t_0,s)\f\|_\infty \le \|P_\varepsilon^{t_0} (\G(t_0,s)\f-\bm\zeta_j)\|_\infty+ \|\bm\zeta_j-\G(t_0,s)\f\|_\infty\le 2\varepsilon.
\end{equation*}
On the other hand, since $P_\varepsilon^{t_0}$ is a contraction in $B_b(\Rd;\Rm)$, it follows that
\begin{align}
&\|P_\varepsilon^{t_0} \G(t_0,s)-\G(t_0,s)\|_{\mathcal{L}(L^1_{\bm{\mu_s}}(\Rd;\Rm);L^1_{\bm{\mu_t}}(\Rd;\Rm))}\notag\\
\le &2\|\G(t_0,s)\|_{\mathcal{L}(L^1_{\bm{\mu_s}}(\Rd;\Rm);L^1_{\bm{\mu_t}}(\Rd;\Rm))}\le 2.
\label{2}
\end{align}
Thus, estimates \eqref{1}, \eqref{2} and the Riesz-Thorin interpolation theorem yield that
\begin{equation}\label{3}
\|P_\varepsilon^{t_0} \G(t_0,s)-\G(t_0,s)\|_{\mathcal{L}(L^p_{\bm{\mu_s}}(\Rd;\Rm);L^p_{\bm{\mu_t}}(\Rd;\Rm))}\le 2\varepsilon^{1-1/p},
\end{equation}
for any $1<p<+\infty$. Letting $\varepsilon\to 0$ in estimate \eqref{3} yields the claim since $\G(t_0,s)$ can be approximated by the operator $P_\varepsilon^{t_0} \G(t_0,s)$ which has range finite.
\end{proof}

\section{Examples}
In this section we provide some examples of operators which satisfy our assumptions and to which our results can be applied.

\begin{example}{\rm
Let $\A$ be as in \eqref{op_A} with
\begin{eqnarray*}
q_{ij}^k(t,x):=\omega_{ij}^k(t)(1+|x|^2)^{h_{ij}^k},\qquad\;\, b^k_i(t,x):= -\gamma_i^k(t)x_i(1+|x|^2)^{\ell_i^k}
\end{eqnarray*}
and
\begin{eqnarray*}
c_{hk}(t,x):=d_{hk}(t)(1+|x|^2)^{\sigma_{hk}}
\end{eqnarray*}
for any $i,j=1, \ldots, d$ and $h, k=1, \ldots, m$.
Let us assume that
\begin{hyp}\label{h_0}
\begin{enumerate}[\rm(i)]
\item
for any $i,j=1, \ldots, d$ and $h, k=1, \ldots, m$, the functions $\omega_{ij}^k, \gamma_i^k$ and $d_{hk}$ belong to $C^{\alpha/2}_{\rm loc}(I)$, $\omega_{ij}^k=\omega_{ji}^k$, $h_{ij}^k=h_{ji}^k$, the coefficients $h_{ij}^k,\ell_i^k, \sigma_{hk}$ are nonnegative and $\inf_I\gamma_{i}^k>0$;
\item
the functions $d_{ij}$ are positive for $i\neq j$, negative for $i=j$ and $\sigma_{ij}<\sigma_{ii}$ for any $i\neq j$;
\item for any $k=1, \ldots, m$, $\min_{i=1,\ldots,d} h_{ii}^k\ge \max_{j\neq i}h_{ij}^k$ and 
$$\nu_k:=\inf_I\bigg(\min_{i=1,\ldots,d} \omega_{ii}^k(t)- \max_{i=1,\ldots,d}\big(\sum_{j\neq i}(\omega_{ij}^k(t))^2\big)^{\frac{1}{2}}\bigg)>0;$$
\item
$1+\max_{i=1,\ldots,d} \{\sigma_{kk}, \ell_i^k\}> \max_{i=1,\ldots,d}h_{ii}^k$, for any $k=1, \ldots, m$.
\end{enumerate}
\end{hyp}
Under Hypotheses \ref{h_0}, all the assumptions in Theorem \ref{bbthe} are satisfied hence it can be applied.
To check Hypothesis \ref{hyp-base}(ii) we can write
\begin{align*}
\langle Q^k(t,x)\zeta,\zeta\rangle=& \sum_{i=1}^d \omega_{ii}^k(t)(1+|x|^2)^{h_{ii}^k}\zeta_i^2+\sum_{j\neq i}\omega_{ij}^k(t)(1+|x|^2)^{h_{ij}^k}\zeta_i\zeta_j\\
\ge & \bigg(\min_i(\omega_{ii}^k)(1+|x|^2)^{\min_{i} h_{ii}^k}\!-\!\max_{i} \bigg(\sum_{j\neq i}(\omega_{ij}^k)^2\bigg)^{\frac{1}{2}}(1+|x|^2)^{\max_{i\neq j}h_{ij}^k}\bigg)|\zeta|^2\\
\ge & (1+|x|^2)^{\max_{i\neq j}h_{ij}^k}\bigg(\min_{i=1,\ldots,d} \omega_{ii}^k(t)- \max_{i=1,\ldots,d}\big(\sum_{j\neq i}(\omega_{ij}^k(t))^2\big)^{\frac{1}{2}}\bigg)|\zeta|^2\\
=: &\,\mu^k(t,x)|\zeta|^2
\end{align*}
for every $\zeta\in\Rd$, and Hypothesis \ref{h_0}(iii) guarantees that the infimum of $\mu^k$ in $I\times \Rd$ is positive for any $k=1, \ldots, m$.
Clearly Hypotheses \ref{hyp-base}(iii) and (iv) are immediate consequences of Hypothesis \ref{h_0}(ii).
Choosing ${\bm\varphi(x)}=\varphi(x){\bm \one}:=(1+|x|^2){\bm \one}$, for every $x \in \Rd$, we get
\begin{align*}
(\A(t)\varphi\bm\one)_k(x)&= 2\sum_{i=1}^d\omega_{ii}^k(t)(1+|x|^2)^{h_{ii}^k}-2\sum_{i=1}^d\gamma_i^k(t)x_i^2(1+|x|^2)^{\ell_i^k}\\
&+\sum_{j=1}^md_{kj}(t)(1+|x|^2)^{\sigma_{kj}+1}
\end{align*}
for every $x\in\Rd$
and from Hypothesis \ref{h_0}(iv) we can prove that there exists two positive constant $a_k,c_k$ such that $(\A(t)\varphi{\bm \one})_k\le a_k-c_k\varphi$, thus Hypothesis \ref{Lya3} (hence Hypothesis \ref{hyp-base}(iv)) is satisfied too.
In addition, since for any $h\neq k$ the functions $c_{hk}$ are nonnegative, the evolution operator $\G(t,s)$ associated to $\A(t)$ is well-defined in $\mathcal{L}(C_b(\Rd;\Rm))$ and it is positive as stated in Proposition \ref{pro}.

Now, we are interested in finding conditions on the coefficients of $\A(t)$ which ensures compactness of $\G(t,s)$ in $C_b(\Rd;\Rm)$ as obtained in Theorems \ref{thm_comp1} and \ref{thm_comp_2}. To this aim, besides Hypotheses \ref{h_0}(i)-(iii) we assume that $\max_{i=1,\ldots,d} h_{ii}^k <1+\max_{i=1,\ldots,d}\ell_i^k$ for any $k=1, \ldots, m$ and that $\sum_{i=1}^m d_{ki}(t)\le 0$ for any $k=1, \ldots, m$. In this case Hypothesis \ref{Lya2}(i) is satisfied with $\psi_k=\varphi$ for any $k=1, \ldots,m$. In addition, being $\sum_{i=1}^m c_{ki}(t,x)\le (1+|x|^2)^{\sigma_{kk}}\sum_{i=1}^m d_{ki}(t)\le 0$, Theorem \ref{thm_comp1} can be applied.\\
On the other hand,  if we assume that $\tau_k:=\max_{i=1,\ldots,d}\{\sigma_{kk},\ell_i^k\}>0$ then Hypothesis \ref{hyp_comp_2}(i) is satisfied with $\varphi(x)=1+|x|^2$ and $h_k(x)=c_1^kx^{1+\tau_k}-c_2^k$ for some positive constants $c_i^k$ $(i=1,2)$. Now we claim that if
\begin{equation}\label{varsaw}
\max_{i=1,\ldots,d} \ell_i^k>1+\max_{i,j=1,\ldots,d}\{\sigma_{kk}, h_{ij}^k-2\}, \qquad\;\, k=1, \ldots, m,
\end{equation}
then the functions $w_k(x)= 1+\frac{1}{1+|x|^2}$, ($k=1,\ldots, m$) are such that Hypothesis \ref{hyp_comp_2}(ii) is satisfied for any $\mu\in \R$, hence Theorem \ref{thm_comp_2} can be applied.
Indeed we can write
\begin{align}\label{tram}
(\mathcal{A}_k(t)w_k)(x)+c_{kk}(t,x)w_k(x)=&-2 \sum_{i=1}^d \omega_{ii}^k(t)(1+|x|^2)^{h_{ii}^k-2}\notag\\
& +2 \sum_{i=1}^d \gamma_i^k(t)x_i^2(1+|x|^2)^{\ell_i^k-2}\notag\\
&+8\sum_{j\neq i}\omega_{ij}^k(t)x_ix_j(1+|x|^2)^{h_{ij}^k-3}\notag\\
&+ d_{kk}(t)(1+|x|^2)^{\sigma_{kk}}\bigg(1+\frac{1}{1+|x|^2}\bigg).
\end{align}
Now, if \eqref{varsaw} is satisfied, the leading part in the right-hand side of \eqref{tram} is given by the term containing the drift coefficients which, as it is easily seen, blows up at infinity. Thus it is clear that we can find $R>0$ such that $\mathcal{A}_kw_k+c_{kk}w_k-\mu w_k$ is positive in $I\times (\Rd\setminus B_R)$ for any $\mu \in \R$.
Consequently, the assumption \eqref{varsaw} is also a sufficient condition in order that neither $C_0(\Rd;\Rm)$, nor $L^p(\Rd;\Rm)$ ($1\le p<+\infty$) are  preserved by the action of $\G(t,s)$ (see Theorems \ref{thm_c0} and \ref{thm_lp}(i)).

Now, we are interested in finding conditions in order that the space $C_0(\Rd;\Rm)$ is preserved by $\G(t,s)$. To this aim, we prove that assuming
\begin{equation}\label{metro}
\max_{i=1,\ldots,d}\{h_{ii}^k-1,\sigma_{kk}\}>\max_{i=1,\ldots,d}\{\max_{j=1,\ldots,d} h_{ij}^k-1,\ell_i^k, \max_{j\neq k}\sigma_{kj}\},
\end{equation}
for any $k=1,\ldots, m$, then we can find $\lambda_0>0$ and $[a,b]\subset I$ such that the function $\vv(x)=\frac{1}{1+|x|^2}\bm\one$ satisfies $\lambda_0 \vv-\A(t) \vv\ge 0$ in $[a,b]\times \Rd$. Indeed, a straightforward computation shows that
\begin{align*}
\lambda_0 v_k(x)\!-\!((\A(t) \vv)(x))_k &= \lambda_0\frac{1}{1+|x|^2}+2\sum_{i=1}^d\omega_{ii}^k(t)(1+|x|^2)^{h_{ii}^k-2}\\
&-8 \sum_{i,j=1}^d \omega_{ij}^k(t)x_ix_j(1+|x|^2)^{h_{ii}^k-3}\!-\!2\sum_{i=1}^d \gamma_i^k x_i^2(1+|x|^2)^{\ell_{i}^k-2}\\
& -\sum_{j=1}^m d_{kj}(t)(1+|x|^2)^{\sigma_{kj}-1}
\end{align*}
for any $k=1, \ldots, m$ and $(t,x)\in I\times \Rd$. Now, arguing as before, if \eqref{metro} is satisfied the function $\lambda_0 v_k(x)-((\A(t) \vv)(x))_k$ tends to $+\infty$ as $|x|\to +\infty$ uniformly with respect to $t\in [a,b]$, for any $[a,b]\subset I$. Hence we can find $\lambda_0>0$ such that $\lambda_0 \vv-\A(t) \vv\ge 0$ in $[a,b]\times \Rd$.

In order to deduce the invariance of $L^p(\Rd;\Rm)$, let us compute $\kappa_C$ which is a function which bounds from above the quadratic form associated to $C$ in $[a,b]\times \Rd$. We can write
\begin{align*}
\langle C(t,x)\zeta,\zeta\rangle =& \langle {\rm diag}C(t,x) \zeta,\zeta\rangle+\langle (C(t,x)-{\rm diag}C(t,x))\zeta,\zeta\rangle\\
\le & -\min_{i=1,\ldots,m} |c_{ii}(t,x)||\zeta|^2+ \Lambda_D(t)(1+|x|^2)^{\max_{i\neq j}\sigma_{ij}}\\
\le & -\min_{i=1,\ldots,m} |d_{ii}(t)|(1+|x|^2)^{\min_{i=1,\ldots,m} \sigma_{ii}}+ \Lambda_D(t)(1+|x|^2)^{\max_{i\neq j}\sigma_{ij}}
\end{align*}
where $\Lambda_D(t)$ is any positive function which bounds from above the quadratic form associated to the matrix $((1-\delta_{hk})d_{hk}(t))_{h,k}$.
Hence, we deduce that $\langle C(t,x)\zeta,\zeta\rangle\le \kappa_C(t,x)|\zeta|^2$ for any $(t,x)\in [a,b]\times \Rd$ where
\begin{eqnarray*}
\kappa_C(t, x)= -(\min_{i=1,\ldots,m} |d_{ii}(t))|(1+|x|^2)^{\min_{i=1,\ldots,m} \sigma_{ii}}+ \Lambda_D(t)(1+|x|^2)^{\max_{i\neq j}\sigma_{ij}},
\end{eqnarray*}
for any $t\in [a,b]$ and $x\in \Rd$. Moreover, since
\begin{align*}
{\rm div}\bm\gamma^k(t,x)=-\sum_{i=1}^d\bigg(&\gamma_i^k(t)(1+|x|^2)^{\ell_i^k}+ 2\ell_i^k\gamma_i^k(t) x_i^2(1+|x|^2)^{\ell_i^k-1}\\
&\;\,+2h_{ii}^k\omega_{ii}^k(t)(1+|x|^2)^{h_{ii}^k-1}\\
&\;\,+4 \sum_{j=1}^d h_{ij}^k(h_{ij}^k-1)\omega_{ij}^k(t)x_jx_i(1+|x|^2)^{h_{ij}^k-2}\bigg),
\end{align*}
we deduce that $\Gamma_{[a,b]}$ is finite (see \eqref{gammaab}) if, for example, $\sigma_{ii}>\max_{s,j,k}\{\ell_s^k,h_{sj}^k-1\}$ for any $i=1,\ldots, m$.
In this case also estimate \eqref{p2} is satisfied, hence Theorem \ref{thm_lp}(ii) and (iii) can be applied. Consequently the space $L^p(\Rd;\Rm)$, $p\ge 1$ turns out to be invariant under $\G(t,s)$.

It is quite easy to see that the functions $\mu_k$, $r_k$, $\rho_0$ and $\rho_1$ defined in Hypotheses \ref{gra_est} are such that
\begin{eqnarray*}
\mu_k(t,x)\simeq |x|^{2\min_i h_{ii}},\quad r_k(t,x)\simeq|x|^{2\min_i \ell_{i}^k}, \quad c_{kk}(t,x)\simeq |x|^{2\sigma_{kk}}
\end{eqnarray*}
and
\begin{eqnarray*}
\rho_0(t,x)\simeq |x|^{2\max_{h\neq k}\sigma_{hk}},\quad \rho_1(t,x)\simeq |x|^{2\max_{h,k} \sigma_{hk}-1}
\end{eqnarray*}
as $|x|\to +\infty$ for any $t\in J$, $J\subset I$ bounded. Thus, taking account of the sign of each term in the definition of $\sigma_{k,J}$ in \eqref{defi} we conclude that $\sigma_{k,J}$ is bounded in $J\times \Rd$ if, for instance
\begin{equation}\label{gr_ipo}
\max\Big\{\min_{i=1,\ldots,d} \ell_i^k, \sigma_{kk}\Big\}> \max\Big\{2\max_{i\neq k}\sigma_{ki},2\sigma_{kk}-1, \min_{i=1,\ldots,d} h_{ii}^k\Big\}, \qquad\;\, k=1, \ldots, m.
\end{equation}
Assumption \eqref{gr_ipo} allows to apply Theorem \ref{stima1} to conclude that $C^1_b(\Rd;\Rm)$ is invariant under $\G(t,s)$.

To conclude, we provide some conditions in order that the results in Section \ref{sect-inv} can be applied.
Besides Hypotheses \ref{h_0}(i)-(iii) we assume that $\max_{i=1,\ldots,d} h_{ii}^k <1+\max_{i=1,\ldots,d}\ell_i^k$ for any $k=1, \ldots, m$ and that $\sum_{i=1}^m d_{ki}(t)\le 0$ for any $k=1, \ldots, m$. In this case Hypotheses \ref{hyp-base}, \ref{Lya2} and $\ref{Lya3}$ are satisfied. If, in addition there exists $j\in \{1, \ldots, m\}$ such that
\begin{eqnarray*}
\max_{i=1,\ldots,d}\ell_i^j>\max_{i\neq k}\{\sigma_{jj},h_{ik}^j-1\},
\end{eqnarray*}
then we can find $K>0$ such that the function $g:\Rd\to\R$, defined by $g(x)=\frac{1}{1+|x|^2}-K$ for any $x\in\Rd$, is such that all the hypotheses in Lemma \ref{pos}(i) are satisfied and Theorem \ref{thm_esm} can be applied.

On the other hand, under Hypotheses \ref{h_0}(i), (iii), if $\sigma_{ij}=\sigma$ for any $i, j=1, \ldots, m$, $d_{ij}>0$ for any $i\neq j$, $\sum_{j=1}^m d_{ij}(t)=0$ for any $t\in I$, $i=1, \ldots, m$, and $\max_{i=1,\ldots,d} h_{ii}^k <1+\max_{i=1,\ldots,d}\ell_i^k$ for any $k=1, \ldots, m$ then Hypotheses \ref{hyp-base}, \ref{Lya2} and $\ref{Lya3}$ are satisfied as well as that in Lemma \ref{pos}(ii) are satisfied. Indeed in this case the function $\g=\bm \one$ is such that $\A(t)\bm \g\equiv \bm{0}$ in $\Rd$ for any $t\in I$ and consequently Theorem \ref{thm_esm} holds true also in this latter case.

 }
\end{example}

\section{Appendix}
Here, we recall some apriori estimates used in the paper, whose proofs can be obtained arguing exactly as in \cite{AALT}, and a classical approximation result.

\begin{prop}
\label{prop-A1}
Let $\Omega\subset\Rd$ be an open set, $T>s\in I$ and $\uu\in C_b([s,T]\times\overline\Omega;\R^m)\cap C^{1,2}((s,T)\times\Omega;\R^m)$
satisfy the equation $D_t\uu=\bm{\mathcal A}\uu+\g$ in $(s,T)\times\Omega$ for some $\g\in C^{\alpha/2,\alpha}((s,T)\times\Omega;\R^m)$. Further, assume that the function $t\mapsto (t-s)\|\uu(t,\cdot)\|_{C^2_b(\Omega;\R^m)}$ is bounded in $(s,T)$. Then, for any $R_1>0$ and $x_0\in\Omega$, such that $D_{R_1}(x_0)\Subset\Omega$, there exists a positive constant $K_0=K_0(R_1,\lambda_0,s,T)$ such that, for any $t\in (s,T)$,
\begin{align}
&(t-s)\|D^2_x\uu(t,\cdot)\|_{L^{\infty}(D_{R_1}(x_0);\R^m)}+\sqrt{t-s}\,\|J_x\uu(t,\cdot)\|_{L^{\infty}(D_{R_1}(x_0);\R^m)}\notag\\
\leq & K_0(\|\uu\|_{C_b([s,T]\times\overline\Omega;\R^m)}
+\|\g\|_{C^{\alpha/2,\alpha}((s,T)\times\Omega;\R^m)}).
\label{i_e}
\end{align}
\end{prop}

\begin{thm}[Interior estimates]\label{int:ext}
Let $T>s\in I$ and let $\uu\in C^{1+\alpha/2,2+\alpha}((s,T]\times \R^d;\R^m)$ satisfy, in $(s,T]\times\R^d$ the equation $D_t\uu=\A\uu+\mathbf{g}$
for some $\mathbf{g}$ belonging to $C^{\alpha/2,\alpha}_{\text{loc}}((s,T]\times\R^d;\R^m)$. Then for every $r_1, r_2\in (s,T)$, with $r_1<r_2$, and any pair of bounded sets $\Omega_1$ and $\Omega_2$ such that $\Omega_1\Subset\Omega_2$, there exists a positive constant $c$, depending on $\Omega_1$, $\Omega_2$, $r_1$, $r_2$, $T$ and $s$, such that
\begin{equation}\label{Schauder}
\norm{\uu}_{C^{1+\alpha/2,2+\alpha}((r_2,T)\times\Omega_1;\R^m)}\leq c(\norm{\uu}_{C_b((r_1,T)\times\Omega_2;\R^m)}+\norm{\mathbf{g}}_{C^{\alpha/2,\alpha}((r_1,T)\times\Omega_2;\R^m)}).
\end{equation}
\end{thm}

\begin{lemm}
\label{lemma-app}
The characteristic function of any open subset of $\Rd$ is the pointwise limit in $\Rd$ of a sequence $(\vartheta_n)\subset C_b(\Rd)$ such that
$0\le\vartheta_n\le 1$ in $\Rd$ for any $n\in\N$.
\end{lemm}

\begin{proof}
We fix an open set $\Omega$ and, for any $n\in\N$, we denote by $\phi_n\in C_b([0,+\infty))$ any function such that $\phi_n(s)=1$, if $s\ge 1/n$, $\phi_n(s)=0$, if $s\in [0,(2n)^{-1}]$ and
$0\le\phi_n(s)\le 1$ otherwise. Next, we set
\begin{eqnarray*}
\vartheta_n(x)=\phi_n(d(x,\Rd\setminus\Omega)),\qquad\;\,x\in\Rd,
\end{eqnarray*}
where $d(x,\Rd\setminus\Omega)$ denotes the distance of $x$ from $\Rd\setminus\Omega$.
As it is immediately seen, each function $\vartheta_n$ vanishes on $\Rd\setminus\Omega$. On the other hand, if $x\in\Omega$, then $d(x,\Rd\setminus\Omega)>0$.
Therefore, if $n\in\N$ is such that $nd(x,\Rd\setminus\Omega)\ge 1$, then $\vartheta_n(x)=1$. As a byproduct, $\lim_{n\to +\infty}\vartheta_n(x)=1$.
Since, by the choice of the sequence $(\phi_n)$ it holds that $0\le\vartheta_n\le 1$ in $\Rd$, $(\vartheta_n)$ is the sequence we are looking for.
\end{proof}

\end{document}